\documentclass{article}

\usepackage[T1]{fontenc}
\usepackage[utf8]{inputenc}
\usepackage[english]{babel}
\usepackage{amsmath}
\usepackage{amsthm}
\usepackage{amssymb}
\usepackage{mathtools}
\usepackage{tikz}
\usepackage{graphicx}
\usepackage{faktor}
\usepackage{xcolor, colortbl}
\usepackage{makecell}
\usepackage{caption}
\usepackage{subcaption}
\usepackage{comment}
\usepackage{setspace}
\usepackage[12pt]{moresize}
\usepackage{wasysym}
\usepackage{float}
\usepackage[margin=1in]{geometry}
\usepackage{mathrsfs}
\usepackage{url}
\usepackage{authblk}
\usepackage{esint}
\usepackage{bbm}
\usepackage{hyperref}
\usepackage{cleveref}

\hypersetup{hidelinks}

\DeclareMathOperator{\id}{id}
\DeclareMathOperator{\dist}{dist}

\DeclareMathOperator{\Lip}{Lip}

\DeclareMathOperator{\diver}{div}

\DeclareMathOperator{\supp}{supp}

\DeclareMathOperator{\diam}{diam}

\theoremstyle{plain}
\newtheorem{thm}{Theorem}[section]

\newtheorem*{thm*}{Theorem}
\newtheorem{lem}[thm]{Lemma}
\newtheorem{prop}[thm]{Proposition}

\newtheorem{defn}[thm]{Definition}
\theoremstyle{remark}
\newtheorem{rmk}[thm]{Remark}

\newcommand{\fr}{\penalty-20\null\hfill$\blacksquare$}

\numberwithin{equation}{section}

\title{A convex integration scheme for the continuity equation\\ past the Sobolev embedding threshold}

\author{Maria Colombo\footnote{EPFL, Station 8, CH-1015 Lausanne, Switzerland. \textit{Email}: maria.colombo@epfl.ch},\,\, Roberto Colombo\footnote{EPFL, Station 8, CH-1015 Lausanne, Switzerland. \textit{Email}: roberto.colombo@epfl.ch},\,\, Anuj Kumar\footnote{Department of Mathematics, University of California Davis, CA 95616, USA. \textit{Email}: anku@ucdavis.edu}}

\date{}

\begin{document}

\maketitle

\begin{abstract}
We introduce a convex integration scheme for the continuity equation in the context of the DiPerna--Lions theory that allows one to build incompressible vector fields in $C_{t}W^{1,p}_x$ and nonunique solutions in $C_{t} L^{q}_x$ for any $p,q$ with $\frac{1}{p} + \frac{1}{q} > 1 + \frac{1}{d}- \delta$ for some $\delta>0$. This improves the previous bound, corresponding to $\delta=0$, or equivalently $q' > p^*$, obtained in \cite{MoSz2019AnnPDE,MoSa2019,BrueColomboDeLellis21}, a threshold that is critical for those schemes in view of the Sobolev embedding that guarantees that solutions are distributional in the opposite range.

\end{abstract}

\bigskip

\noindent \textbf{MSC:} 35A02; 35Q49; 35D30; 35Q35.\\
\noindent \textbf{Keywords:} Continuity equation; Transport; DiPerna--Lions theory; Convex integration; Nonuniqueness.

\section{Introduction}

In a pioneering work, DiPerna and Lions \cite{DiPernaLions} developed the theory for the solution of the transport equation 
\begin{align}
\partial_t \rho + b \cdot \nabla \rho = 0, \qquad \rho(0, \cdot) = \rho_0 \in L^q,
\label{eqn: transport}
\end{align}
where $\rho \in C([0, T], L^q(\mathbb{R}^d))$ is the density field and $b \in L^\infty([0, T], W^{1, p}(\mathbb{R}^d; \mathbb{R}^d))$ is the vector field with bounded divergence. The transport equation is formally equivalent to the continuity equation
\begin{align*}
\partial_t \rho + \diver(\rho b) = 0,
\end{align*}
when the vector field $b$ is divergence-free. An important consequence of DiPerna--Lions theory is that the solution to the Cauchy problem \eqref{eqn: transport} is unique if the relationship between $p$ and $q$ satisfies
\begin{align}
\frac{1}{p} + \frac{1}{q} \leq 1. 
\label{unique range DL}
\end{align}
In the range \eqref{unique range DL} the solutions are renormalized and Lagrangian, meaning that they can be constructed using the regular Lagrangian flow map associated with the vector field $b$.

Since the theory of DiPerna and Lions, an important question that remained unanswered is whether the uniqueness range \eqref{unique range DL} is sharp, i.e., is it possible to produce an example of nonuniqueness outside this range?
In recent years, tremendous progress has been made on this question. The works \cite{MoSz2019AnnPDE,MoSz2019CalcVar,MoSa2019} provided the first examples of nonuniqueness for the continuity equation when the exponents $q$ and $p$ satisfy 
\begin{align}
\frac{1}{p} + \frac{1}{q} > 1 + \frac{1}{d}. 
\label{ex MoSa19}
\end{align}
There, the velocity $b$ also belongs to $L^{q^\prime}$ for $\rho$ to be a weak solution. 
Their constructions use the method of convex integration, which was originally developed by De Lellis and Sz\'ekelyhidi \cite{de2009euler, de2013dissipative} for the Euler equations. The convex integration method has proven to be highly successful in the realm of fluid dynamics PDEs, from resolving Onsager's conjecture \cite{isett2018proof, buckmaster2019onsager} to producing nonunique solutions to the Navier--Stokes equations \cite{buckmaster2019nonuniqueness, buckmaster2021wild}. Further advances for linear transport using the convex integration method include the temporally intermittent constructions of \cite{MR4199851,cheskidov2022extreme}, which considered the case where the transported density $\rho$ and the velocity field $b$ have spatial integrability exponents in the complement of \eqref{unique range DL}, while being only integrable in time; these works also show that the relevant threshold strongly depends on the class of solutions under consideration.

Recently, \cite{BrueColomboDeLellis21} improved the uniqueness range for positive solutions to the continuity equation, extending the result given by DiPerna and Lions. Their proof relies on the combination of an asymmetric version of the Lusin--Lipschitz type inequality \cite{CrDL} and Ambrosio's superposition principle \cite{Ambrosio04,AmbrCr} which works only for positive densities; see also the previous works \cite{Jabin2010,CaravennaCrippa16,CaravennaCrippa18} which substantially contributed to framing the uniqueness/Lagrangianity results surrounding the problem. The improved range is
\begin{align}
\frac{1}{p} + \frac{1}{q} < 1 + \frac{1}{dq}.
\label{unique range BCD}
\end{align}
In addition to establishing this new range of uniqueness, \cite{BrueColomboDeLellis21} provided examples of nonuniqueness within the class of positive solutions, specifically for exponents that meet the range given by \eqref{ex MoSa19}. Combined with Ambrosio's superposition principle, nonuniqueness in the class of positive densities then yields the existence of a positive Lebesgue measure set of initial conditions for which the vector field admits at least two distinct trajectories; see also \cite{GiriSorella21,PitchoSorella2023,pappalettera2023measure} for further contributions relating Eulerian and Lagrangian viewpoints in incompressible transport. 

A new approach to tackle the problem of nonuniqueness was proposed at the Lagrangian level in \cite{kumar2023nonuniqueness}, which used a self-similar construction to provide examples of divergence-free Sobolev ($W^{1, p}$ with $p < d$) and H\"older continuous ($C^{\alpha}$ with $\alpha < 1$) vector fields whose trajectories are nonunique for a full-measure set of initial conditions. 
However, at the level of the continuity equation, these examples only yield nonuniqueness within the class of measures. The work \cite{brue2024sharp} then constructed vector fields through a fixed point iteration method and showed nonuniqueness of positive density fields. This result sharply matches the uniqueness range established by \cite{BrueColomboDeLellis21} as given in \eqref{unique range BCD}. One of the key advancements enabling the construction of sharp examples was the introduction of spatial heterogeneity to the maximum possible extent in the design of the velocity field. This approach is in contrast to conventional convex integration constructions, where the velocity perturbations are typically periodic on a scale of $1/\lambda$, with $\lambda$ being a large number. 

It is therefore natural to ask how much of this sharp picture can be recovered within convex integration. This question is not only about nonuniqueness, but also about flexibility: roughly speaking, starting from the same initial data $\rho_0$, one can prescribe a suitable profile $\phi(t)$ and construct solutions whose $L^q$-norm remains within $\epsilon$ distance of this profile, namely $\left| \|\rho(t)\|_{L^q} - \phi(t) \right| < \epsilon$.
In this sense, nonuniqueness is only one consequence of the broader flexibility mechanism underlying this technique. Moreover, convex integration is a more natural method for nonlinear PDEs such as the Euler and Navier--Stokes equations where explicit constructions are difficult to achieve. Inspired by this question, \cite{BrCoKu24} developed an heterogeneous version of the convex integration scheme, where the velocity perturbations are not $\lambda$-periodic. They demonstrated nonuniqueness of solutions to the two-dimensional Euler equation with bounded kinetic energy and vorticity in $L^p$ for some $p > 1$, surpassing for the first time the critical scaling of the standard convex integration scheme previously reached in \cite{BrueColombo23,BuckModena24}.

In this paper, we adapt the convex integration scheme proposed by \cite{BrCoKu24} and apply it to the continuity equation, extending our results beyond the critical range given by \eqref{ex MoSa19}. The following theorem gives, to the best of our knowledge, the largest currently available range obtained by a flexible convex integration construction for the linear continuity equation.
For simplicity, we choose to work in dimension $d \geq 3$; however, our argument can be modified to accommodate $d = 2$ with the introduction of certain time-cutoff functions (see, for example, \cite{BrCoKu24}). Our main result is the following:
\begin{thm}
    \label{thm_main} Let $d\geq 3 $ be the dimension and $q> 1$. 
    Then there exists $\delta \coloneqq \delta(d,q)>0$, such that for any $p\geq 1$ that satisfies
    \[1+\frac{1}{d}-\frac{1}{p}-\frac{1}{q}<\delta,\]
    for any continuous function $\phi: [0,1]\rightarrow [0,\infty)$ and any $\epsilon>0$, there exists an incompressible vector field $b\in C\big([0,1],L^{q'}\cap W^{1,p}(\mathbb{T}^{d}; \mathbb{R}^{d})\big)$ and a density $\rho \in C\big([0,1],L^{q}(\mathbb{T}^{d})\big)$ such that
    \begin{equation*}
        \partial_{t}\rho+\diver(\rho b)=0\quad \text{in $(0,1)\times \mathbb{T}^{d}$},\qquad
            \underset{t\in [0,1]}{\max}\,\Big|\lVert \rho(t,\cdot)\rVert_{L^{q}}-\phi(t)\Big|<\epsilon.
    \end{equation*}
\end{thm}
\noindent Since $b$ is Sobolev and incompressible, there exists a unique regular Lagrangian flow $X:[0,1]\times \mathbb{T}^{d}\to \mathbb{T}^{d}$ associated to $b$, with $X(t,\cdot)$ measure preserving at each time (see \cite{AmbrCr}). In particular, we can construct a Lagrangian solution $\tilde{\rho}(t,x)=\rho_{0}(X_{t}^{-1}(x))$ of the continuity equation with vector field $b$, with initial condition $\tilde{\rho}(0,\cdot)=\rho_{0}=\rho(0,\cdot)$ and such that $\lVert \tilde{\rho}(t,\cdot)\rVert_{L^{q}}$ is constant in time. As a consequence, if in Theorem \ref{thm_main} the $L^{q}$-norm profile $\phi$ is not constant and $\epsilon>0$ is chosen sufficiently small, then the two solutions $\rho$ and $\tilde{\rho}$ cannot coincide, and nonuniqueness occurs in the class of $C_{t}L_{x}^{q}$-densities.

The gap $\delta(d,q)$ we obtain in Theorem \ref{thm_main} (see \eqref{formula-gap} for an explicit quantification) is smaller than $(dq')^{-1}$, corresponding to the sharp range for nonnegative solutions \cite{brue2024sharp}. Therefore, the construction in \cite{brue2024sharp} works for a wider range of exponents, but is more rigid. We emphasize that the sharp nonuniqueness result in \cite{brue2024sharp} concerns the class of positive densities, whereas the densities constructed in the present paper may have both signs. We note that the question of the sharp exponent range for sign-changing densities at present remains open. The novelty of the present paper lies instead in the localization of the error-cancellation mechanism: at each iteration step, the perturbation used to cancel the error is localized; see the outline of the proof below. 

Can the construction of flexible solutions with convex integration be pushed to the sharp range for nonnegative solutions, namely obtaining Theorem \ref{thm_main} with $\delta(d,q) = (dq')^{-1}$? What is the sharp range of well-posedness for sign-changing solutions of the transport equation?  
Besides their intrinsic interest, a deeper understanding of such issues could give useful insights towards analogous questions in the context of nonlinear equations of fluid dynamics.
    
\paragraph{Outline of the proof.} 
    Without loss of generality we will work in the critical or subcritical regime, that is, we will assume that
    \begin{equation}\label{defn_alpha}
        \alpha \coloneqq 1+\frac{1}{d}-\frac{1}{p}-\frac{1}{q}\ge0.
    \end{equation}
    \eqref{defn_alpha} is equivalent to $p^{*}\ge q'$ which implies the validity of the Sobolev embedding $W^{1,p}(\mathbb{T}^{d})\hookrightarrow L^{q'} (\mathbb{T}^{d})$. Indeed, convex integration schemes are already known to work in the supercritical regime $\alpha<0$, see \cite{MoSa2019} and \cite{BrueColomboDeLellis21}.

The exotic solution $(\rho, b)$ of Theorem \ref{thm_main} is obtained as the limit of a sequence of smooth approximate solutions. Following the lines of \cite{MoSz2019AnnPDE}, we construct recursively a sequence of smooth triplets $(\rho_{n},b_{n},R_{n})\in C^{\infty}([0,1]\times \mathbb{T}^{d};\mathbb{R}\times \mathbb{R}^{d}\times \mathbb{R}^{d})$ composed of a density $\rho_{n}$, an incompressible vector field $b_{n}$, and an \lq\lq error'' $R_{n}$, solutions of the so-called \lq\lq continuity-defect equation''
$$\partial_{t}\rho_{n}+\diver(\rho_{n}b_{n})=-\diver R_{n}$$
and fulfilling the following convergences in the sense of distributions
\begin{equation*}
    \rho_{n}\rightarrow \rho,\quad b_{n}\rightarrow b,\quad \rho_{n}b_{n}\rightarrow \rho b,\quad R_{n}\rightarrow 0.    
\end{equation*}
The limit vector field $b$ will be divergence-free and $\rho$ will solve $\partial_{t}\rho+\diver(\rho b)=0$ distributionally. To obtain $\rho \in C_{t}L_{x}^{q}$ and $b\in C_{t}W_{x}^{1,p}$, we need to impose suitable quantitative estimates of the type
\begin{equation}\label{eq:iterative_estimatesIntro}
    \lVert \rho_{n+1}-\rho_{n}\rVert_{C_{t}L_{x}^{q}}+\lVert b_{n+1}-b_{n}\rVert_{C_{t}W_{x}^{1,p}}\le \lVert R_{n}\rVert_{C_{t}L_{x}^{1}}^{\gamma_{1}},\qquad \lVert R_{n+1}\rVert_{C_{t}L_{x}^{1}}\le \lVert R_{n}\rVert_{C_{t}L_{x}^{1}}^{\gamma_{2}}  
\end{equation}
for some $\gamma_{1} >0$ and $\gamma_{2}>1$. In this way, as long as $\lVert R_{0}\rVert_{C_{t}L_{x}^{1}}$ is small enough, $\lVert R_{n}\rVert_{C_{t}L_{x}^{1}}$ decreases super-exponentially and $\rho_{n},b_{n}$ are Cauchy sequences in $C_{t}L_{x}^{q}$ and $C_{t}W_{x}^{1,p}$, respectively. 

The main point on which our scheme differs from the constructions already present in the literature is the design of the \lq\lq perturbations'' $\rho_{n+1}-\rho_{n}$, $b_{n+1}-b_{n}$ and of the new error $R_{n+1}$. In order to satisfy \eqref{eq:iterative_estimatesIntro}, the perturbations need to be appropriately estimated in terms of the old error $R_{n}$, and, at the same time, they have to interact quadratically so as to \lq\lq cancel'' $R_{n}$ itself, thus producing a much smaller new error $R_{n+1}$. To reach this goal, in \cite{MoSz2019AnnPDE}, \cite{MoSa2019} and \cite{BrueColomboDeLellis21} the principal terms in the perturbations are designed as the product of a slowly varying coefficient of the size of the old error $R_{n}$ times $\lambda$-periodic and $r$-concentrated building blocks, which themselves are solutions to the transport equation. To fix ideas, suppose that the old error $R_{n}=R_{n}^{\xi}\xi$, with $R_{n}^{\xi}>0$, is a positive shear in some direction $\xi\in \mathbb{S}^{d-1}$. Then we would have
$$\rho_{n+1}-\rho_{n}\approx (R_{n}^{\xi})^{1/q}\theta_{r,\lambda},\quad b_{n+1}-b_{n}\approx (R_{n}^{\xi})^{1/q'}w_{r,\lambda},$$
where $\theta_{r,\lambda}$ and $w_{r,\lambda}$ solve $\partial_{t}\theta_{r,\lambda}+\diver(\theta_{r,\lambda}w_{r,\lambda})=0$, $w_{r,\lambda}$ has direction $\xi$, and both are $r$-concentrated and $\lambda$-periodic. Then, the product of the perturbations 
$$(\rho_{n+1}-\rho_{n})(b_{n+1}-b_{n})\approx R_{n}^{\xi}\theta_{r,\lambda}w_{r,\lambda}$$
reconstructs the old error up to a fast oscillating function, and the cancellation is obtained after a suitable choice of an anti-divergence. With the aforementioned constraints, the choice of $\theta_{r,\lambda}$ and $w_{r,\lambda}$ as \lq\lq traveling balls'' in \cite{BrueColomboDeLellis21} leads to the range of nonuniqueness in \eqref{ex MoSa19}. 

To go past the range \eqref{ex MoSa19}, which is critical in \cite{BrueColomboDeLellis21} in view of scaling considerations, we use the same construction of building blocks, but we do not impose the $\lambda$-periodicity in space at each time and instead restore it in a time-averaged fashion. Here, we will show a possible way to implement this idea. Again, for simplicity, assume that $R_{n}=R_{n}^{\xi}\xi$ is a positive shear in direction $\xi$, this time with $\xi=(1,\lambda^{-1},\dots,\lambda^{-(d-1)})$, so that a point in the torus moving in the direction $\xi$ has a closed orbit $\left\{s\xi:s\in \mathbb{R}\right\}$ which is approximately $\lambda$-periodic. The principal perturbations will then be given by
$$\rho_{n+1}-\rho_{n}\approx \lVert R_{n}^{\xi}(t,\cdot)\rVert_{L^{1}}^{1/2}\theta_{r(t)}(\cdot-z(t)\xi),\qquad b_{n+1}-b_{n}\approx \lVert R_{n}^{\xi}(t,\cdot)\rVert_{L^{1}}^{1/2}w_{r(t)}(\cdot-z(t)\xi).$$
Here, the building blocks are traveling balls traversing the torus in the direction $\xi$, with a speed given by
$$\dot{z}(t)\approx \frac{1}{R_{n}^{\xi}(t,z(t)\xi)}.$$
In this way, letting $\tau>0$ denote a small time period needed to complete one tour of the torus, the average
$$\fint_{t}^{t+\tau}(\rho_{n+1}-\rho_{n})(b_{n+1}-b_{n})\approx R_{n}^{\xi}(t,\cdot)\fint_{t}^{t+\tau}\theta_{r}w_{r}$$
reconstructs the old error $R_{n}$ up to an almost $\lambda$-periodic function obtained by \lq\lq translating'' the building blocks over their orbit. Finally, the replacement of the product of the perturbations with its time average will be justified by the addition of a lower order term in the density perturbation.
Note that the role of the torus in the present construction is not only to provide a compact environment without boundary effects. The rational directions introduced above generate periodic orbits on $\mathbb{T}^{d}$, allowing spatial averages to be replaced by averages along a single such orbit. Although the estimates are local in nature, the periodic domain provides a clean compact setting for this averaging mechanism.

The paper is organized as follows: in Section \ref{sec_pertthm} we state the perturbation Theorem \ref{thm_pert}, providing the inductive step in the construction of the approximating sequence described above, and we use it to prove Theorem \ref{thm_main}. The rest of the paper is then devoted to the proof of Theorem \ref{thm_pert}, which requires several old and new constructions and estimates. In Section \ref{sec_fundamental-lemmas} we present the fundamental lemmas underlying our new error cancellation mechanism, which we call \lq\lq path periodicity''. Then, in Section \ref{sec_constructions} we collect all the constructions which are required to define the new approximate solution produced by the perturbation theorem. Finally, Section \ref{sec_estimates} is dedicated to the derivation of all the required quantitative estimates and, ultimately, to the proof of Theorem \ref{thm_pert}. 

\section{The perturbation theorem}\label{sec_pertthm}
First of all, we introduce a class of approximate solutions that satisfy suitable quantitative estimates. 
\begin{defn}
    Given $a_{1},a_{2}, T, \eta >0$, we denote by $\mathscr{A}_{a_{1},a_{2}}(T,\eta)$ the set of all smooth triplets $(\rho,b,R):\mathbb{R}\times \mathbb{T}^{d}\to \mathbb{R}\times \mathbb{R}^{d}\times \mathbb{R}^{d}$, solutions of the continuity-defect equation
    \begin{equation*}
	\label{CDE}
	\begin{cases}
		\partial_{t}\rho+\diver(\rho b)=-\diver R,\\
		\diver b=0,
	\end{cases}\tag{CDE}		
\end{equation*} 
    and satisfying the following conditions:
    \begin{itemize}
        \item [(i)] $\underset{t\in [0,T]}{\max}\lVert R(t,\cdot)\rVert_{L^{1}}\le \eta$.
        \item [(ii)] $\underset{t\in [0,T]}{\max}\Big(\lVert \rho(t,\cdot)\rVert_{L^{q}}+\lVert b(t,\cdot)\rVert_{W^{1,p}}\Big)\le 1-2\eta^{a_{1}}$.
        \item [(iii)] $\underset{t\in [0,T]}{\max}\Big(\lVert \nabla_{t,x}\rho(t,\cdot)\rVert_{L^{q}}+\lVert \nabla_{t,x}b(t,\cdot)\rVert_{W^{1,p}}\Big)\le \eta^{-a_{2}}$.
    \end{itemize}
\end{defn}
\noindent Next, we state a perturbation result and use it to prove Theorem \ref{thm_main}.
\begin{thm}\label{thm_pert} Let $d\ge 3$ and $q>1$. Then there exists $\delta=\delta(d,q)>0$ such that, for every $p\ge 1$ for which
$0\le 1+\frac{1}{d}-\frac{1}{p}-\frac{1}{q}<\delta$ the following holds:
there exist $\bar{\eta}\in (0,1)$ small enough and parameters $a_{1}>0, a_{2}>1, a_{3} >0$, $a_{4}>1$ such that, for every $T\in [1,2]$ and every $(\rho, b, R)\in \mathscr{A}_{a_{1},a_{2}}(T,\eta)$ with $\eta\in (0,\bar{\eta}]$, we can find $(\tilde{\rho}, \tilde{b}, \tilde{R})\in \mathscr{A}_{a_{1},a_{2}}(T-\eta^{a_{3}},\eta^{a_{4}})$ for which
\begin{equation*}
    \underset{t\in [0,T-\eta^{a_{3}}]}{\max}\Big(\lVert \tilde{\rho}(t,\cdot)-\rho(t,\cdot)\rVert_{L^{q}}+\lVert \tilde{b}(t,\cdot)-b(t,\cdot)\rVert_{W^{1,p}}\Big)\le \eta^{a_{1}}.
\end{equation*}

\end{thm}

\smallskip
\begin{proof}[Proof of Theorem \ref{thm_main}] 
Let $\delta=\delta(d,q)>0$ be as in Theorem \ref{thm_pert} and let $p\ge 1$ be such that 
$\alpha \coloneqq 1+\frac{1}{d}-\frac{1}{p}-\frac{1}{q}<\delta$. Up to increasing $p$, we may further assume $\alpha\ge 0$, so that the Sobolev embedding $W^{1,p}(\mathbb{T}^{d})\hookrightarrow L^{q'}(\mathbb{T}^{d})$ holds true.
We seek a solution of the continuity equation whose $L^{q}$-norm is $\epsilon$-close to the prescribed continuous function $\phi:[0,1]\to[0,\infty)$, uniformly in time. The case where $\phi$ is constant is trivial, since the uniform density solves the continuity equation with zero velocity field. We shall therefore assume that $\phi$ is not constant on $[0,1]$. Moreover, by first extending $\phi$ to the whole real line and then approximating it uniformly by smooth functions, we may assume that $\phi$ is defined on $\mathbb{R}$ and is smooth. 
In addition, the linearity of the equation further reduces the argument to the case $$\max_{t\in [0,2]}\phi(t)\le 1/2.$$ 
We fix a positive integer $k\in \mathbb{N}\setminus \{0\}$ to be chosen sufficiently large in the end, and define
\begin{equation*}
    \rho_{0}(t,x) \coloneqq \phi(t)\frac{\cos(2\pi k x_{1})}{\lVert \cos(2\pi x_{1})\rVert_{L^{q}}},\qquad b_{0} \coloneqq 0,\qquad R_{0}= -\partial_{t}\phi(t)\frac{\sin(2\pi kx_{1})e_{1}}{2\pi k \lVert \cos(2\pi x_{1})\rVert_{L^{q}}}.
\end{equation*}
By construction, $(\rho_{0}, b_{0}, R_{0})$ solves (CDE). In addition, we have the following: 
\begin{gather*}
     \eta \coloneqq \underset{t\in [0,2]}{\max}\lVert R_{0}(t,\cdot)\rVert_{L^{1}}= C_{1}(\phi)k^{-1}>0,\\
     \underset{t\in [0,2]}{\max}\Big(\lVert \rho_{0}(t,\cdot)\rVert_{L^{q}}+\lVert b_{0}(t,\cdot)\rVert_{W^{1,p}}\Big)\le 1/2,\\
     \underset{t\in [0,2]}{\max}\Big(\lVert \nabla_{t,x}\rho_{0}(t,\cdot)\rVert_{L^{q}}+\lVert \nabla_{t,x}b_{0}(t,\cdot)\rVert_{W^{1,p}}\Big)\le C_{2}(\phi)(1+k).
\end{gather*}
Let us take $\bar{\eta}>0$ and $a_{1}>0, a_{2}>1, a_{3}>0, a_{4}>1$ as in Theorem \ref{thm_pert}. Since $a_{2}>1$, provided that $k$ is chosen large enough, we have $(\rho_{0},b_{0},R_{0})\in \mathscr{A}_{a_{1},a_{2}}(2,\eta)$ and $\eta \le \bar{\eta}$. Moreover, since $a_{4}>1$, again by taking $k$ sufficiently large we have
\begin{equation*}
    \overunderset{\infty}{j=0}{\sum}\eta^{a_{1}a_{4}^{j}}<\epsilon, \qquad \overunderset{\infty}{j=0}{\sum}\eta^{a_{3}a_{4}^{j}}<1.
\end{equation*}
An iterative application of Theorem \ref{thm_pert} yields a sequence $\left\{(\rho_{n},b_{n},R_{n})\right\}_{n\in \mathbb{N}}$ of triplets such that, for every $n\in \mathbb{N}$:
\begin{gather*}
    (\rho_{n},b_{n},R_{n}) \in \mathscr{A}_{a_{1},a_{2}}\left(2-\overunderset{n-1}{j=0}{\sum}\eta^{a_{3}a_{4}^{j}}, \eta^{a_{4}^{n}}\right),\\
    \overunderset{n-1}{j=0}{\sum}\,\underset{t\in[0,1]}{\max}\Big(\lVert \rho_{j+1}(t,\cdot)-\rho_{j}(t,\cdot)\rVert_{L^{q}}+\lVert b_{j+1}(t,\cdot)-b_{j}(t,\cdot)\rVert_{W^{1,p}}\Big)\le \overunderset{n-1}{j=0}{\sum}\eta^{a_{1}a_{4}^{j}}.
\end{gather*}
In particular $\left\{\rho_{n}\right\}_{n\in \mathbb{N}}$ and $\left\{b_{n}\right\}_{n\in \mathbb{N}}$ are Cauchy sequences in the spaces $C([0,1],L^{q})$ and $C([0,1],W^{1,p})$, respectively, and we have the following strong convergences:
\begin{equation*}
    \rho_{n}\stackrel{C_{t}L_{x}^{q}}{\longrightarrow} \rho \coloneqq \rho_{0}+\overunderset{\infty}{j=0}{\sum}(\rho_{j+1}-\rho_{j}),\qquad 
    b_{n}\stackrel{C_{t}W_{x}^{1,p}}{\longrightarrow} b \coloneqq b_{0}+\overunderset{\infty}{j=0}{\sum}(b_{j+1}-b_{j}),\qquad
    R_{n}\stackrel{C_{t}L_{x}^{1}}{\longrightarrow} 0.
\end{equation*}
In addition, the Sobolev embedding $W^{1,p}\hookrightarrow L^{q'}$ gives also the strong convergence of $\rho_{n}b_{n}$ to $\rho b$ in $C_{t}L_{x}^{1}$, and we can pass to the limit in the definition of solution to the continuity-defect equation to get
\begin{equation*}
    \partial_{t}\rho+\diver(\rho b)=0 \qquad \text{distributionally in $(0,1)\times \mathbb{T}^{d}$.}
\end{equation*}
Finally, 
\begin{equation*}
    \underset{t\in [0,1]}{\max}\Big|\lVert \rho(t,\cdot)\rVert_{L^{q}}-\phi(t)\Big|\le \underset{t\in [0,1]}{\max}\lVert \rho(t,\cdot)-\rho_{0}(t,\cdot)\rVert_{L^{q}}\le \overunderset{\infty}{j=0}{\sum}\,\underset{t\in[0,1]}{\max}\lVert \rho_{j+1}(t,\cdot)-\rho_{j}(t,\cdot)\rVert_{L^{q}}\le \overunderset{\infty}{j=0}{\sum}\eta^{a_{1}a_{4}^{j}}<\epsilon.
\end{equation*}
This concludes the proof. 
\end{proof}

\section{The fundamental lemmas}\label{sec_fundamental-lemmas}
This section is devoted to the development of the underlying geometric structure of our work. The possibility of applying the convex integration technique beyond the critical relation $p^{*}=q'$ is due to the use of a new, more effective method to exploit intermittency, first introduced in \cite{BrCoKu24}. This new mechanism, which we call \lq\lq path periodicity'' is based on the idea of recovering spatial periodicity in a time-average fashion. 
\paragraph{Path periodicity.}
We consider a large positive integer $\lambda \in \mathbb{N}$. To $\lambda$ we can associate a special vector $\xi \in \mathbb{R}^{d}$ defined as follows:
\begin{equation*}\label{defvectorperiodicity}
	\xi =\left(1,\lambda^{-1},\dots,\lambda^{1-d}\right).
\end{equation*}
Since $\xi$ has rational components, a point moving in the torus with velocity $\xi$ has a closed orbit $\{u\xi:u\in \mathbb{R}\}$, its lowest period being exactly
\begin{equation}\label{deflengthorbit}
	L \coloneqq \lambda^{d-1}.
\end{equation}
The \lq\lq stripe'' of width $r>0$ around this orbit will be denoted by
\begin{equation*}\label{defstripe}
	S(\xi, r) \coloneqq  \left\{x\in \mathbb{T}^{d}: \dist\left(x, \{u\xi:u\in \mathbb{R}\}\right)\le r\right\}.
\end{equation*}
Notice that when $r\lesssim_{d} \lambda^{-1}$ is small enough, $S(\xi, r)$ is a tubular neighborhood of $\{u\xi:u\in \mathbb{R}\}$.

As a preliminary, we construct an appropriate continuous partition of unity in the torus which is parameterized at the points of $\{u\xi:u\in \mathbb{R}\}$. Although this has already been done in \cite[Section 4.7]{BrCoKu24}, we still include a proof hereafter for the reader's convenience.

\begin{lem}\label{lempartitionunityperiodic}
	There exists a bump function $U\in C^{\infty}(\mathbb{T}^{d})$ such that $U\ge 0$, $\supp U \subset B_{c_{d}\lambda^{-1}}$, $\int_{\mathbb{T}^{d}}U = 1$, and for which $$\fint_{0}^{L}U(x-u\xi)\,du=1\qquad \text{for every $x\in \mathbb{T}^{d}$}. $$
\end{lem}
\begin{proof}
    We first consider a nonnegative function $\Omega\in C^{\infty}(\mathbb{T}^{d})$ with $\int_{\mathbb{T}^{d}}\Omega=1$ and such that $\Omega$ is a constant in $B_{c_{d}\lambda^{-1}/2}$ and $\Omega =0$ in $\mathbb{T}^{d}\setminus B_{c_{d}\lambda^{-1}}$. Next we define 
    $$\widetilde{\Omega}(x) \coloneqq  \fint_{0}^{L}\Omega(x-u\xi)\,du,$$
    which is clearly invariant under translations in the $\xi$-direction, and we choose $c_{d}$ sufficiently large so that $\widetilde{\Omega}>0$ in the whole $\mathbb{T}^{d}$. Then, the function $U \coloneqq  \Omega/\widetilde{\Omega}$ has all the properties stated in the lemma. 
\end{proof}

The continuous partition of unity given by Lemma \ref{lempartitionunityperiodic} is a useful tool to prove several statements in the context of path periodicity. For instance, we can use it to get the following lemma, which states that the integral of a function over $\mathbb{T}^{d}$ can be replaced with its average over the orbit $\{u\xi:u\in \mathbb{R}\}$ up to a small error of order $\lambda^{-1}$.
\begin{lem}\label{lem_recoveringL1normcrawling}
	For every $f\in \Lip(\mathbb{T}^{d})$, we have $\left|\fint_{0}^{L}f(u\xi)\,du-\int_{\mathbb{T}^{d}}f\right|\le c_{d}\Lip(f)\lambda^{-1}$.
\end{lem}
\begin{proof}
	Consider $U$ as in Lemma \ref{lempartitionunityperiodic}. Then, since $$\int_{\mathbb{T}^{d}}U(x-u\xi)\,dx=1 \quad \text{for every $u\in [0,L)$},\qquad \fint_{0}^{L}U(x-u\xi)\,du=1 \quad \text{for every $x\in \mathbb{T}^{d}$},$$
	we may rewrite 
	\begin{align*}
		\left|\fint_{0}^{L}f(u\xi)\,du-\int_{\mathbb{T}^{d}}f(x)\,dx\right|&
        =
        \left|\int_{\mathbb{T}^{d}}\fint_{0}^{L}U(x-u\xi)f(u\xi)\,du\,dx-\int_{\mathbb{T}^{d}}\fint_{0}^{L}U(x-u\xi)f(x)\,du\,dx
        \right|
        \\&\le \fint_{0}^{L}\int_{\mathbb{T}^{d}}U(x-u\xi)|f(u\xi)-f(x)|\,dx\,du\\
		&\le \Lip(f)\diam(\supp U)\fint_{0}^{L}\int_{\mathbb{T}^{d}}U(x-u\xi)\,dx\,du\\
		&=c_{d}\Lip(f)\lambda^{-1},
	\end{align*} 
    as desired.
\end{proof}
In the next lemma, we wish to present a computation of the $L^{s}$-norm of the \lq\lq trace'' left by a small bump, of varying size, as it traverses the torus with velocity $\xi$. 
\begin{lem}\label{lemnormtraceleftroundtorus}
    Let $r:[0,L]\rightarrow (0,\infty)$ be such that $0<r_{\min}\le r(u)\le r_{\max}\le  c_{d}\lambda^{-1}$ so that the stripe $S(\xi,r_{\max})$ is a tubular neighborhood of $\{u\xi:u\in \mathbb{R}\}$. Then, calling $\bar{B}(u) \coloneqq B_{r(u)}(u \xi)$, we have
	\begin{equation*}
		\Big \lVert \int_{0}^{L}r(u)^{-d}\mathbbm{1}_{\bar{B}(u)}\,du \Big \rVert_{L^{s}}\le c_{d}\lambda^{(d-1)/s}r_{\max}^{(d-1+s)/s}r_{\min}^{-d}\qquad \text{for every $s\in [1,\infty]$}.
	\end{equation*}
\end{lem}
\begin{proof}
	We start from the pointwise estimate
	\[\left|\int_{0}^{L}r(u)^{-d}\mathbbm{1}_{\bar{B}(u)}(x)\,du\right|\le r_{\min}^{-d}\mathbbm{1}_{S(\xi,r_{\max})}(x)\mathscr{L}^{1}\big(\left\{u\in [0,L): |x-u\xi|\le r_{\max}\right\}\big).\]
	Then we observe that since $S(\xi,r_{\max})$ is a tubular neighborhood of $\{u\xi:u\in \mathbb{R}\}$, for each point $x\in S(\xi, r_{\max})$, we have
	\[\mathscr{L}^{1}\big(\left\{u\in [0,L): |x-u\xi|\le r_{\max}\right\}\big)\le c_{d}r_{\max}.\]
	Using also that $\mathscr{L}^{d}\left(S(\xi,r_{\max})\right)\le c_{d}\lambda^{d-1}r_{\max}^{d-1}$, we finally obtain  
	\begin{equation*}
		\Big \lVert \int_{0}^{L}r(u)^{-d}\mathbbm{1}_{\bar{B}(u)}(\cdot)\,du \Big \rVert_{L^{s}}\le \mathscr{L}^{d}\left(S(\xi,r_{\max})\right)^{1/s}2r_{\max}r_{\min}^{-d} \le c_{d}\lambda^{(d-1)/s}r_{\max}^{(d-1+s)/s}r_{\min}^{-d}.
	\end{equation*}
\end{proof}

\paragraph{Anti-divergence operators.} Given a zero-averaged smooth function $g\in C_{0}^{\infty}(\mathbb{T}^{d})$, we say that a smooth vector field $v\in C^{\infty}(\mathbb{T}^{d}; \mathbb{R}^{d})$ is an \textit{anti-divergence} of $g$ if $\diver v = g$. A standard way to invert the divergence operator makes use of the Poisson equation. Namely, for every $g\in C_{0}^{\infty}(\mathbb{T}^{d})$, we consider the unique $u\in C_{0}^{\infty}(\mathbb{T}^{d})$ solving $\Delta u = g$, and we write $u \coloneqq  \Delta^{-1}g$. Then, the \textit{standard anti-divergence} of $g$ is defined as $\nabla \Delta^{-1}g$.
The regularity theory of elliptic PDEs gives the following (See \cite[Lemma 2.2]{MoSz2019AnnPDE} and the references therein):
\begin{lem}
	\label{lemstandardantidiv}
	For every $g\in C_{0}^{\infty}(\mathbb{T}^{d})$, we have the following bounds:
	\begin{equation*}
		\lVert \nabla^{k}\nabla \Delta^{-1}g\rVert_{L^{s}}\le c_{d,k,s}\lVert \nabla^{k}g\rVert_{L^{s}}\qquad \text{for every $k\in \mathbb{N}$ and every $s\in [1,\infty]$}.
	\end{equation*}
\end{lem}
\noindent The endpoint cases $s=1$ and $s=\infty$ are not exceptional here, in contrast with the usual Calder\'{o}n--Zygmund setting. Indeed, $\nabla \Delta^{-1}$ is an operator of order $-1$, whose convolution kernel is integrable on the torus; a similar remark applies to Proposition \ref{prop-bogovskii} below. 

In the sequel we need an efficient way to invert the divergence of concentrated functions, gaining in the estimates a factor of the order of the size of their support. We will use the following useful formula from \cite{bogovskii1979solution}:
\begin{defn}
    Let $\varphi \in C_{c}^{\infty}(\mathbb{R}^{d})$ be such that $\int_{\mathbb{R}^{d}}\varphi =1$. For every $g\in C_{c}^{\infty}(\mathbb{R}^{d})$ with $\int_{\mathbb{R}^{d}}g=0$, the Bogovski{\u{\i}} anti-divergence of $g$ relative to the kernel $\varphi$ is given by
    $$\mathscr{B}_{\varphi}(g)(x) \coloneqq  \int_{\mathbb{R}^{d}}g(y)\frac{x-y}{|x-y|}\int_{0}^{\infty}\varphi\left(x+s\frac{x-y}{|x-y|}\right)ds\,dy.$$
\end{defn}
\noindent It is not difficult to show that $\mathscr{B}_{\varphi}(g)$ is indeed a well-defined smooth compactly supported anti-divergence of $g$. In order to take advantage of the size of the support of $g$ it is convenient to consider $r$-rescaled versions of the Bogovski{\u{\i}} anti-divergence:
$$\mathscr{B}_{\varphi, r}(g)(x) \coloneqq  r\mathscr{B}_{\varphi}(g(r\cdot))\left(\frac{x}{r}\right).$$
The results in \cite{bogovskii1979solution} and a simple scaling argument yield the following:
\begin{prop}\label{prop-bogovskii}
    Let $\varphi\in C_{c}^{\infty}(B_{1})$ be such that $\int_{\mathbb{R}^{d}}\varphi =1$. Then, for every $r>0$ and every $g\in C_{c}^{\infty}(B_{r})$ with zero average, $\mathscr{B}_{\varphi,r}(g)$ is a smooth anti-divergence of $g$, and $\supp \mathscr{B}_{\varphi,r}(g)\subset B_{r}$. Moreover, 
    $$\lVert \nabla^{k}\mathscr{B}_{\varphi,r}(g)\rVert_{L^{s}}\le c_{d,k,s}\lVert\nabla^{k} g\rVert_{L^{s}}r \qquad \text{for every $k\in \mathbb{N}$ and every $s\in [1,\infty]$}.$$
\end{prop}

Proposition \ref{prop-bogovskii}, coupled with Lemma \ref{lempartitionunityperiodic} gives the next result, which represents one of the keys of the whole scheme. It states that the function obtained by \lq\lq crawling'' a concentrated bump along the orbit $\{u\xi:u\in \mathbb{R}\}$ admits an anti-divergence of size $\lambda^{-1}$.
\begin{lem}\label{lemimpantidivcrawling}
	Let $D \in C^{\infty}([0,L]\times \mathbb{T}^{d})$ be nonnegative and such that $$\int_{\mathbb{T}^{d}}D(u,x)\,dx=1,\quad \supp D(u,\cdot) \subset B_{\lambda^{-1}}(u\xi)\qquad \text{for every $u\in [0,L]$.}$$
	Then, there exists a vector field $v\in C^{\infty}(\mathbb{T}^{d};\mathbb{R}^{d})$ such that $\diver v = \fint_{0}^{L}D(u,\cdot)\,du -1$ and $\lVert v\rVert_{L^{1}}\le c_{d}\lambda^{-1}$.
\end{lem}
\begin{proof}
	Given $U$ as in Lemma \ref{lempartitionunityperiodic}, we can write, for every $x\in \mathbb{T}^{d}$,
	\[\fint_{0}^{L}D(u,x)\,du -1=\fint_{0}^{L}\big(D(u,x)-U(x-u\xi)\big)du.\] 
	Now notice that for each $u\in [0,L]$, the smooth map $g^{u}(x) \coloneqq D(u,x)-U(x-u\xi)$
	has zero average, $\lVert g^{u}\rVert_{L^{1}}\le 2$ and $\diam(\supp g^{u})\le c_{d}\lambda^{-1}$. Then, by Proposition \ref{prop-bogovskii}, we may construct $v^{u}\in C^{\infty}(\mathbb{T}^{d};\mathbb{R}^{d})$ such that $\diver v^{u}=g^{u}$ and $$\lVert v^{u}\rVert_{L^{1}}\le c_{d}\lVert g^{u}\rVert_{L^{1}}\diam (\supp g^{u})\le c_{d}\lambda^{-1}.$$
    To conclude, we may take
	\[v(x) \coloneqq \fint_{0}^{L}v^{u}(x)\,du.\]
	In fact, by linearity,
	\[\diver v(x) =\fint_{0}^{L}\diver v^{u}(x)\,du=\fint_{0}^{L}g^{u}(x)\,du=\fint_{0}^{L}D(u,x)\,du -1,\]
	and moreover, 
	\[\lVert v \rVert_{L^{1}}\le \fint_{0}^{L}\lVert v^{u}\rVert_{L^{1}}\,du\le  c_{d}\lambda^{-1}.\]
\end{proof}

We end this section by providing a convenient way to invert the divergence of the product of two functions, one of which is faster oscillating than the other. This lemma originates in \cite[Lemma 2.3]{MoSz2019AnnPDE}. Here we adapt the original formulation to allow for other choices of anti-divergence of $\diver v$, including a proof for the reader's convenience.
\begin{lem}\label{lem-smart-anti-divergence}
    Given $f\in C^{\infty}(\mathbb{T}^{d})$ and $v\in C^{\infty}(\mathbb{T}^{d};\mathbb{R}^{d})$, the field $\mathcal{R}(f,v)\in C^{\infty}(\mathbb{T}^{d};\mathbb{R}^{d})$, defined as
    \begin{equation}\label{def-smart-anti-divergence}
        \mathcal{R}(f,v) \coloneqq  fv-\nabla \Delta^{-1}\left(\nabla f \cdot v + \int_{\mathbb{T}^{d}}f\diver v\right)
    \end{equation}
    is an anti-divergence of $f\diver v-\int_{\mathbb{T}^{d}}f\diver v$ and satisfies the following bounds:
    $$\lVert \nabla^{k}\mathcal{R}(f,v)\rVert_{L^{s}}\le c_{d,k,s}\lVert f\rVert_{C^{k+1}}\lVert v\rVert_{W^{k,s}}\qquad \text{for every $k\in \mathbb{N}$ and every $s\in [1,\infty]$}.$$
\end{lem}
\begin{proof}
    The fact that $\diver \mathcal{R}(f,v)=f\diver v-\int_{\mathbb{T}^{d}}f\diver v$ comes from a direct computation. Regarding the estimate, we can use the bounds on the standard anti-divergence $\nabla \Delta^{-1}$ from Lemma \ref{lemstandardantidiv} to get
    \begin{align*}
        \lVert \nabla ^{k}\mathcal{R}(f,v)\rVert_{L^{s}}&\le c_{k}\lVert f\rVert_{C^{k}}\lVert v\rVert_{W^{k,s}}+ \Big\lVert\nabla^{k}\nabla \Delta^{-1}\left(\nabla f \cdot v + \int_{\mathbb{T}^{d}}f\diver v\right)\Big\rVert_{L^{s}}\\
        &\le c_{k}\lVert f\rVert_{C^{k}}\lVert v\rVert_{W^{k,s}}+ c_{d,k,s}\lVert \nabla f \rVert_{C^{k}}\lVert v\rVert_{W^{k,s}}\\
        &\le c_{d,k,s}\lVert f\rVert_{C^{k+1}}\lVert v\rVert_{W^{k,s}}.
    \end{align*}
\end{proof}
\section{Construction of the new approximate solution}\label{sec_constructions}
The goal of this section is to define the new approximate solution $(\tilde{\rho},\tilde{b},\tilde{R})$ starting from the old one $(\rho, b, R)$, as required by Theorem \ref{thm_pert}. This will require several preliminary constructions. 
\paragraph{A basis of good vectors.}  First of all, we need to consider an entire basis of $\mathbb{R}^{d}$ made of vectors with the same \lq\lq path periodicity properties'' described in the previous section as $\xi=(1,\lambda^{-1},\dots,\lambda^{-(d-1)})$. A simple way to define such a basis is by taking advantage of the cyclic permutations of the components of $\xi$. Indeed, the vectors 
$$\xi_{1} \coloneqq \left(1,\lambda^{-1},\dots,\lambda^{-(d-1)}\right),\quad \xi_{2} \coloneqq \left(\lambda^{-(d-1)},1,\dots,\lambda^{-(d-2)}\right),\quad \dots,\quad\xi_{d} \coloneqq \left(\lambda^{-1},\lambda^{-2},\dots,1\right)$$
satisfy the relations $|\xi_{i}\cdot \xi_{j}-\delta_{ij}|\le c_{d}\lambda^{-1}$ meaning that, for large values of $\lambda$, the set $\left\{\xi_{j}\right\}_{j=1}^{d}$ forms an almost orthonormal basis of $\mathbb{R}^{d}$.
Actually, our scheme requires vectors of $\mathbb{R}^{d}$ to be decomposed with nonnegative coefficients with respect to some generating system. That is why we double the basis $\left\{\xi_{j}\right\}_{j=1}^{d}$ by introducing also
\begin{equation*}
\xi_{d+j} \coloneqq -\xi_{j}\quad \text{for every $j\in \left\{1,\dots,d\right\}$.}
\end{equation*}
Then, every vector $a= \overunderset{d}{j=1}{\sum}a^{j}\xi_{j}\in \mathbb{R}^{d}$ admits the representation $$a= \overunderset{2d}{j=1}{\sum}\hat{a}^{j}\xi_{j},$$ where $\hat{a}^{j}=a^{j,+}$ and $\hat{a}^{d+j}=a^{j,-}$ for every $j\in \left\{1,\dots,d\right\}$.

The next lemma shows that provided $d\ge 3$, once suitably translated in $\mathbb{T}^{d}$, the orbits corresponding to different vectors $\xi_{i}$ and $\xi_{j}$, for $1\le i<j\le 2d$ can be taken nonintersecting.
\begin{lem}\label{lemma:separation-curves}
    Let $d\ge 3$. Then there are points $\bar{x}_{1},\dots, \bar{x}_{2d}\in \mathbb{T}^{d}$ such that $$|(\bar{x}_{i}+t\xi_{i})-(\bar{x}_{j}+s\xi_{j})|\ge c_{d}L^{-\frac{2}{d-2}}\quad \text{for every $1\le i<j\le 2d$ and every $t,s \in [0,L]$}.$$
\end{lem}
\begin{proof}
For every $1\le i<j\le 2d$ and every $r>0$, we call $$S_{i,j} \coloneqq  \left\{t\xi_{i}-s\xi_{j}:t,s\in [0,L]\right\},\qquad K_{i,j}(r) \coloneqq  \left\{x\in \mathbb{T}^{d}: \exists t,s \in [0,L] \text{ s.t. } |t\xi_{i}-s\xi_{j}-x|< r\right\}.$$
We first observe that $S_{i,j}$ is a $2$-surface with area $\mathcal{H}^{2}(S_{i,j})\le c_{d}L^{2}$. Then, since $K_{i,j}(r)$ is the $r$-enlargement of $S_{i,j}$, we deduce that $\mathscr{L}^{d}(K_{i,j}(r))\le c_{d}L^{2}r^{d-2}$. Hence, provided that $r\le c_{d}L^{-2/(d-2)}$ for a sufficiently small dimensional constant $c_{d}$, we can make sure to find points $\bar{x}_{1},\dots,\bar{x}_{2d}$ such that $\bar{x}_{i}-\bar{x}_{j}\notin K_{i,j}(r)$, for every $1\le i<j\le 2d$, as desired.
\end{proof}

\paragraph{A preliminary regularization.} In order to keep the $L^{\infty}$-norms of the approximate solution $(\rho,b,R)$ and its derivatives under control, we perform a preliminary regularization via convolution, standard in convex integration schemes. We first decompose $R$ with respect to the generating system $\left\{\xi_{j}\right\}_{j=1}^{2d}$ with nonnegative coefficients:
\begin{equation*}
    R=\overunderset{2d}{j=1}{\sum}\hat{R}^{j}\xi_{j},\qquad \hat{R}^{j}\ge 0\quad \text{for every $j\in \left\{1,\dots,2d\right\}$}.
\end{equation*}
Then, given a standard space-time convolution kernel $\varphi \in C_{c}^{\infty}$ and fixed a parameter $\ell>0$ to be chosen later in Section \ref{sec_estimates}, we define
\begin{equation*}
    \rho_{\ell} \coloneqq \rho*\varphi_{\ell},\quad b_{\ell} \coloneqq b*\varphi_{\ell}, \quad (\rho b)_{\ell} \coloneqq  (\rho b)*\varphi_{\ell}, \quad R_{\ell} \coloneqq R*\varphi_{\ell}, \quad R_{\ell}^{j} \coloneqq  \hat{R}^{j}*\varphi_{\ell}.
\end{equation*}
Clearly, by the linearity of the convolution, 
\begin{equation*}
    R_{\ell}=\overunderset{2d}{j=1}{\sum}R_{\ell}^{j}\xi_{j},\qquad R_{\ell}^{j}\ge 0\quad \text{for every $j\in \left\{1,\dots,2d\right\}$}.
\end{equation*}

In the next lemma we gather all the corrective estimates related to the replacement of $(\rho, b, R)$ with its regularization. They follow from standard results on convolutions; see for instance \cite[Lemma 5.1]{BrueColomboDeLellis21} for a proof.
\begin{lem}
    The triplet $\big(\rho_{\ell}, b_{\ell},R_{\ell}+(\rho b)_{\ell}-\rho_{\ell}b_{\ell}\big)$ solves (CDE). Moreover, the following estimates hold:
    \begin{subequations}
        \begin{gather}
         \lVert R_{\ell}\rVert_{C_{t}L_{x}^{1}}\le \lVert R\rVert_{C_{t}L_{x}^{1}},\quad  \lVert \partial_{t}^{h}\nabla_{x}^{k}R_{\ell}\rVert_{C_{t,x}}\label{boundRconvolution}\le c_{d,h,k}\ell^{-d-h-k}\lVert R\rVert_{C_{t}L_{x}^{1}}\quad \text{for every $h,k\in \mathbb{N}$};\\
        \lVert \rho_{\ell}-\rho\rVert_{C_{t}L_{x}^{q}}\le c_{d,q}\ell \lVert \nabla_{t,x}\rho\rVert_{C_{t}L_{x}^{q}},\quad \lVert b_{\ell}-b\rVert_{C_{t}W_{x}^{1,p}}\le c_{d,p}\ell \lVert \nabla_{t,x}b\rVert_{C_{t}W_{x}^{1,p}};\label{bound-correction-convolution-density&field}\\
        \lVert (\rho b)_{\ell}-\rho_{\ell}b_{\ell}\rVert_{C_{t}L_{x}^{1}}\le c_{d}\ell^{2}\lVert \nabla_{t,x}\rho\rVert_{C_{t}L_{x}^{q}}\lVert \nabla_{t,x}b\rVert_{C_{t}L_{x}^{q'}}\label{bound-correction-convolution-product}.
        \end{gather}
    \end{subequations}
\end{lem}

\paragraph{The trajectory of the building blocks.} The major contributions to the perturbations $\tilde{\rho}-\rho$ and $\tilde{b}-b$ will be made by the sum of $2d$ \lq\lq projectiles'' (the building blocks) traversing the torus in the direction of $\xi_{1},\dots,\xi_{2d}$, with nonhomogeneous size and speed locally dictated by the form of the coefficients $R_{\ell}^{1},\dots, R_{\ell}^{2d}$. In this paragraph, we wish to design the precise speed and size of such projectiles. 

We may assume that $R\not\equiv 0$, and denote for simplicity
\begin{equation*}
    \eta \coloneqq \lVert R\rVert_{C_{t}L_{x}^{1}}>0.
\end{equation*}
For every $j\in \left\{1,\dots,2d\right\}$, we define the following objects. First, we introduce the radius function $r_{j}:\mathbb{R}\times\mathbb{T}^{d}\rightarrow (0,\infty)$ prescribing the size of the support of the projectile depending on its position:
    \begin{equation*}
        r_{j}(t,x) \coloneqq  \bar{r}\left(R_{\ell}^{j}(t,x)+\eta\right)^{q'/d}.
    \end{equation*}     
Here $\bar{r}>0$, representing the typical concentration scale, will be chosen in Section \ref{sec_estimates}. Then, we prescribe the position of the center of the projectile $x_{j}:\mathbb{R}\rightarrow \mathbb{T}^{d}$ as the solution of the following ODE:
    \begin{align}
        \begin{cases}
            \dot{x}_{j}(t)=\sigma r_{j}(t,x_{j}(t))^{-d/q'}\xi_{j},\\
            x_{j}(0)=\bar{x}_{j}.
        \end{cases}
    \label{the ode}
    \end{align}
    Here, $\{\bar{x}_{j}\}_{j=1}^{2d}$ are the points given by Lemma \ref{lemma:separation-curves}, and the tuning parameter $\sigma>0$ will be chosen in Section \ref{sec_estimates}. 
    Since $r_{j}$ is uniformly bounded above, the speed $|\dot{x}_{j}|$ will be uniformly bounded away from zero. Since each trajectory $x_{j}$ has a closed orbit of length $L|\xi_{j}|$, the return time $\tau_{j}:\mathbb{R}\rightarrow (0,\infty)$ is well-defined by the formula
    \begin{equation}\label{eq-fundamental-tauj}
        x_{j}(t+\tau_{j}(t))=x_{j}(t)+L\xi_{j}\qquad \text{for every $t\in \mathbb{R}$.}
    \end{equation}
    An expression for $\tau_{j}(t)$ can be deduced from \eqref{the ode}, after integration: 
    \begin{equation}\label{eq-formula-tauj}
        \tau_{j}(t)=\bar{r}^{d/q'}L\sigma^{-1}\left(\frac{1}{L|\xi_{j}|}\int_{t}^{t+\tau_{j}(t)}R_{\ell}^{j}(s,x_{j}(s))|\dot{x}_{j}(s)|\,ds+\eta\right).
    \end{equation}
    
    In the following lemma we gather for future reference some estimates on these newly defined quantities.
\begin{lem}
    For every $j\in \left\{1,\dots,2d\right\}$, let $r_{j}, x_{j}$ and $\tau_{j}$ be the functions defined above. Then:
    \begin{itemize}
        \item [(i)] $r_{j}\in [r_{\min},r_{\max}]$, where 
        \begin{equation}\label{defn_rminrmax}
        r_{\min}=\bar{r}\eta^{q'/d},\qquad r_{\max}=c_{d,q}\bar{r}\ell^{-q'}\eta^{q'/d}.
    \end{equation}
    The partial derivatives of $r_{j}$ are given by
    \begin{equation}\label{formula_derivativesrj}
        \partial_{t}r_{j}=\frac{q'}{d}\bar{r}^{d/q'}r_{j}^{1-d/q'}\partial_{t}R_{\ell}^{j},\qquad  \nabla r_{j}=\frac{q'}{d}\bar{r}^{d/q'}r_{j}^{1-d/q'}\nabla R_{\ell}^{j}.
    \end{equation}
    The following estimates hold: 
    \begin{align}\label{eq-bounds-derivatives-rj}
    |\partial_t r_j|, |\nabla r_j| \leq c_{d, q} \bar{r} \ell^{-q^\prime - d - 1} \eta^{q^\prime/d}.
\end{align}
\item[(ii)] The speed $|\dot{x}_{j}|$ is bounded above and below by
    \begin{equation}\label{boundsspeed}
        c_{d,q}\bar{r}^{-d/q'}\ell^{d}\sigma \eta^{-1}\le|\dot{x}_{j}|\le |\xi_{j}|\bar{r}^{-d/q'}\sigma \eta^{-1}  
    \end{equation}
    \item[(iii)] Assuming that $\ell^{-d-1}(\lambda^{-1}+\bar{r}^{d/q'}L\ell^{-d}\sigma^{-1}\eta)\le 1$, then $\tau_{j}$ is bounded above and below by
    \begin{equation}\label{bounds-tauj}
        \bar{r}^{d/q'}L\sigma^{-1}\eta\le \tau_{j}\le c_{d,q} \bar{r}^{d/q'}L\sigma^{-1}\eta.
    \end{equation}
    Moreover, the following estimate holds:
    \begin{equation}\label{upperbound-dettauj}
        |\partial_{t}\tau_{j}|\le c_{d,q}\bar{r}^{d/q'}L\ell^{-d-1}\sigma^{-1}\eta.
    \end{equation}
    \end{itemize}
\end{lem}
\begin{proof}
    The bounds on $r_{j}$ in \eqref{defn_rminrmax} come from \eqref{boundRconvolution}. Formulas \eqref{formula_derivativesrj} are obtained with a direct computation, while \eqref{eq-bounds-derivatives-rj} follows from \eqref{defn_rminrmax} and \eqref{boundRconvolution}. 
    The bounds on $|\dot{x}_{j}|$ in \eqref{boundsspeed} can be derived applying \eqref{boundRconvolution} in the expression for $\dot{x}_{j}$ from \eqref{the ode}. 
    
    In order to prove some bounds for $\tau_{j}$ we can use \eqref{eq-fundamental-tauj} combined with the upper and lower bounds on $|\dot{x}_{j}|$ from \eqref{boundsspeed}, and get
    \begin{equation*}
        \bar{r}^{d/q'}L\sigma^{-1}\eta\le \tau_{j}\le c_{d,q}\bar{r}^{d/q'}L\ell^{-d}\sigma^{-1}\eta.
    \end{equation*}
    Then, to promote the upper bound to the one in \eqref{bounds-tauj}, we exploit formula \eqref{eq-formula-tauj}:
    \begin{align*}
        \tau_{j}(t)&= \bar{r}^{d/q'}L\sigma^{-1}\left(\frac{1}{L|\xi_{j}|}\int_{t}^{t+\tau_{j}(t)}R_{\ell}^{j}(s,x_{j}(s))|\dot{x}_{j}(s)|\,ds+\eta\right)\\
        &\le \bar{r}^{d/q'}L\sigma^{-1}\left(\fint_{0}^{L}R_{\ell}^{j}(t,\bar{x}_{j}+u\xi_{j})\,du+\lVert \partial_{t}R_{\ell}^{j}\rVert_{C_{t,x}}\tau_{j}(t)+\eta\right)\\
        &\le \bar{r}^{d/q'}L\sigma^{-1}\left(\lVert R_{\ell}^{j}(t,\cdot)\rVert_{L^{1}}+c_{d}\lVert \nabla R_{\ell}^{j}\rVert_{C_{t,x}}\lambda^{-1}+\lVert \partial_{t}R_{\ell}^{j}\rVert_{C_{t,x}}\tau_{j}(t)+\eta\right)\\
        & \le c_{d,q}\bar{r}^{d/q'}L\sigma^{-1}\eta\left(1+\ell^{-d-1}(\lambda^{-1}+\bar{r}^{d/q'}L\ell^{-d}\sigma^{-1}\eta)\right)\le c_{d,q}\bar{r}^{d/q'}L\sigma^{-1}\eta.
    \end{align*}
    In the second inequality, we froze the time variable and made a change of variables; in the third, we applied Lemma \ref{lem_recoveringL1normcrawling} to the function $R_{\ell}^{j}(t,\cdot)$; in the fourth, we used \eqref{boundRconvolution} to bound the derivatives of $R_{\ell}^{j}$ and in the end we used hypothesis $\ell^{-d-1}(\lambda^{-1}+\bar{r}^{d/q'}L\ell^{-d}\sigma^{-1}\eta)\le 1$.

    Finally, let us bound $|\partial_{t}\tau_{j}|$. From \eqref{eq-formula-tauj} we deduce that
    $$\partial_{t}\tau_{j}=\bar{r}^{d/q'}\sigma^{-1}\frac{|\dot{x}_{j}(t)|}{|\xi_{j}|}\left(R_{\ell}^{j}(t+\tau_{j}(t), x_{j}(t))-R_{\ell}^{j}(t, x_{j}(t))\right).$$
    Therefore, $|\partial_{t}\tau_{j}|\le c_{d}\bar{r}^{d/q'}\sigma^{-1}|\dot{x}_{j}(t)|\tau_{j}(t)\lVert \partial_{t}R_{\ell}^{j}\rVert_{C_{t,x}}$,
and \eqref{upperbound-dettauj} is deduced from \eqref{boundRconvolution}, \eqref{boundsspeed} and \eqref{bounds-tauj}.
\end{proof}
\begin{rmk}\label{rmk-constraints-parameters}
    From here onward we will impose the following constraint on the parameters:
\begin{equation}\label{eq-constraint-parameters1}
     \ell^{-d-1}(\lambda^{-1}+\bar{r}^{d/q'}\lambda^{d-1}\ell^{-d}\sigma^{-1}\eta)\le 1.
\end{equation}
    In this way, the estimates \eqref{bounds-tauj} and \eqref{upperbound-dettauj} hold true.
    In addition, we will assume 
    \begin{equation}\label{eq-constraint-parameters2}
        \bar{r}\ell^{-q'}\eta^{q'/d}\le c_{d,q}\lambda^{-\frac{2d-2}{d-2}},
    \end{equation}
    so that $r_{\max}\le c_{d}\lambda^{-\frac{2(d-1)}{d-2}}$ holds true in view of \eqref{defn_rminrmax}, and by the choice of the points $\{\bar{x}_{j}\}_{j=1}^{2d}$, the orbits of $x_{i}$ and $x_{j}$ will be separated at least by $2r_{\max}$, for every $i\neq j$ (see Lemma \ref{lemma:separation-curves}). \fr
    \end{rmk}

\paragraph{The building blocks.} 
For every $j\in \left\{1,\dots,2d\right\}$ we consider functions $\bar{\theta}_{j}\in C_{c}^{\infty}(\mathbb{R}^{d})$ and $\bar{w}_{j}\in C_{c}^{\infty}(\mathbb{R}^{d};\mathbb{R}^{d})$ with the following properties:
\begin{equation}\label{eq_fundamental-properties-blobs}
	\begin{cases}
		\supp \bar{\theta}_{j}\subset B_{1/2},\\
		\bar{\theta}_{j}\ge 0,\\
		\int \bar{\theta}_{j}=1,
	\end{cases}\quad\quad 
	\begin{cases}
		\supp \bar{w}_{j}\subset B_{1},\\
		\bar{w}_{j}(x)=\xi_{j}\quad\quad \text{for every $x\in B_{1/2}$},\\
		\diver \bar{w}_{j}=0.
	\end{cases}
\end{equation}
An explicit construction of these objects can be found in \cite[Section 4]{BrueColomboDeLellis21}.
Notice that $\bar{\theta}_{j}\bar{w}_{j}=\bar{\theta}_{j}\xi_{j}$ and in particular, the following fundamental identity holds: 
\begin{equation}\label{eq_basic-cancellation-blobs}
    \diver(\bar{\theta}_{j}\bar{w}_{j})=\nabla \bar{\theta}_{j} \cdot \xi_{j}.
\end{equation}
Our building blocks are then obtained from $\bar{\theta}_{j}, \bar{w}_{j}$ after a suitable translation, rescaling and introduction of a time variable. For every $j\in \left\{1,\dots,2d\right\}$, 
\begin{equation*}
    \theta_{j}:\mathbb{R}\times \mathbb{R}^{d}\rightarrow \mathbb{R},\qquad w_{j}:\mathbb{R}\times \mathbb{R}^{d}\rightarrow \mathbb{R}^{d}
\end{equation*}
are defined as follows:
\begin{equation}\label{def thetaj wj}
    \theta_{j}(t,x) \coloneqq  r_{j}(t,x_{j}(t))^{-d/q}\bar{\theta}_{j}\left(\frac{x-x_{j}(t)}{r_{j}(t,x_{j}(t))}\right),\quad w_{j}(t,x) \coloneqq  r_{j}(t,x_{j}(t))^{-d/q'}\bar{w}_{j}\left(\frac{x-x_{j}(t)}{r_{j}(t,x_{j}(t))}\right).
\end{equation}
Due to the constraints imposed in Remark \ref{rmk-constraints-parameters}, we may assume that $r_{j}\le 1/2$, which implies that both $\theta_{j}(t,\cdot)$ and $w_{j}(t,\cdot)$ have support contained in a ball of radius less than $1/2$. Therefore, with a slight abuse of notation, we shall denote by $\theta_{j}(t,\cdot), w_{j}(t,\cdot)$ also their periodic extension on the torus,
the concentration of the supports being preserved in $\mathbb{T}^{d}$. 
As it will recur often in the sequel, let us also introduce a notation for the following function:
\begin{equation}\label{def-Dj}
    D_{j}(t,x) \coloneqq  r_{j}(t,x_{j}(t))^{-d}\bar{\theta}_{j}\left(\frac{x-x_{j}(t)}{r_{j}(t,x_{j}(t))}\right).
\end{equation}
Observe that $D_{j}(t,\cdot)$ has unit integral, and is concentrated in the ball of radius $r_{j}(t,x_{j}(t))/2$ with center $x_{j}(t)$.

In the following lemma we collect for future reference some basic properties of the building blocks, as well as some estimates of their derivatives.
\begin{lem}\label{lemma-buildingblocks}
	Let $\left\{\theta_{j}\right\}_{j=1}^{2d}$ and $\left\{w_{j}\right\}_{j=1}^{2d}$ be the building blocks defined in \eqref{def thetaj wj}. Then:
	\begin{itemize}
		\item[(i)] For every $t\in \mathbb{R}$, we have:
            \begin{gather*}
                \supp \theta_{j}(t,\cdot)\subset B_{\frac{r_{j}(t,x_{j}(t))}{2}}(x_{j}(t)),\quad \supp w_{j}(t,\cdot)\subset B_{r_{j}(t,x_{j}(t))}(x_{j}(t));\\
				\int_{\mathbb{T}^{d}}\theta_{j}(t,\cdot)w_{j}(t,\cdot)=\xi_{j},\quad \int_{\mathbb{T}^{d}}|\theta_{j}(t,\cdot)w_{j}(t,\cdot)|=|\xi_{j}|;\\
    \theta_{j}(t,\cdot)\ge 0,\quad \diver w_{j}(t,\cdot)=0.
            \end{gather*}       
        \item[(ii)] Supports are well-separated, i.e.
        \begin{equation}\label{eq_separation-supports}
            \theta_{i}(t,\cdot)w_{j}(t,\cdot)\equiv 0\qquad \text{whenever $i \neq j$}.
        \end{equation}
\item[(iii)] For every $t\in \mathbb{R}$ and every $k\in \mathbb{N}$ and  $s\in [1,\infty]$, we have the scaling laws
    \begin{subequations}
        \begin{align}
			\lVert \nabla^{k}\theta_{j}(t,\cdot)\rVert_{L^{s}}\label{scalingbbdensity}&= c_{d,k,s}r_{j}(t,x_{j}(t))^{-k-\frac{d}{q}+\frac{d}{s}},\\
			\lVert \nabla^{k}w_{j}(t,\cdot)\rVert_{L^{s}}\label{scalingbbfield}&= c_{d,k,s}r_{j}(t,x_{j}(t))^{-k-\frac{d}{q'}+\frac{d}{s}}.
		\end{align}
    \end{subequations}
    Moreover, we have the following estimates ($\alpha$ is defined in \eqref{defn_alpha}):
    \begin{subequations}
        \begin{align}
            \lVert \partial_{t}\theta_{j}(t,\cdot)\rVert_{L^{q}}\label{eq-bound-det-thetaj}&\le c_{d,q}r_{j}(t,x_{j}(t))^{-1}(1+|\partial_{t}r_{j}(t,x_{j}(t))|+|\nabla r_{j}(t,x_{j}(t))|)(1+|\dot{x}_{j}(t)|),\\
            \lVert \partial_{t}w_{j}(t,\cdot)\rVert_{W^{1,p}}\label{eq-bound-det-wj}&\le c_{d,p,q}r_{j}(t,x_{j}(t))^{-1-d\alpha}(1+|\partial_{t}r_{j}(t,x_{j}(t))|+|\nabla r_{j}(t,x_{j}(t))|)(1+|\dot{x}_{j}(t)|).
        \end{align}
    \end{subequations}
\item[(iv)] The following identity holds:
\begin{equation}\label{eq-cancellation-buildingblocks}
\begin{split}
    \partial_{t}\theta_{j}+\sigma\diver(\theta_{j}w_{j})&=\bar{r}^{d/q'}\nabla R_{\ell}^{j}(t,x_{j}(t))\cdot \dot{x}_{j}(t)D_{j}(t,x)\\
    &\quad\,+\bar{r}^{d/q'}\partial_{t}R_{\ell}^{j}(t,x_{j}(t))D_{j}(t,x) \\
    &\quad\,-\diver\left(\frac{q'}{d}\bar{r}^{d/q'}\left(\partial_{t}R_{\ell}^{j}(t,x_{j}(t))+\nabla R_{\ell}^{j}(t,x_{j}(t))\cdot \dot{x}_{j}(t)\right)D_{j}(t,x)(x-x_{j}(t))\right).
\end{split}
\end{equation}
	\end{itemize}
\end{lem} 
\begin{proof}
    Point (i) is a direct consequence of the properties of $\bar \theta_j, \bar w_j$ in \eqref{eq_fundamental-properties-blobs}. \eqref{eq_separation-supports} follows from the choice of parameters made in \eqref{eq-constraint-parameters2}, as explained in Remark \ref{rmk-constraints-parameters}. Point (iii) is a consequence of the scaling properties of $L^{s}$-norms. Finally, formula \eqref{eq-cancellation-buildingblocks} can be obtained with a straightforward computation, exploiting the cancellation coming from \eqref{eq_basic-cancellation-blobs}. 
\end{proof}

\paragraph{The perturbations.} We are now in the position to define the perturbations of our scheme:
\begin{equation}\label{formulaperturbations}
    \tilde{\rho}-\rho= (\rho_{\ell}-\rho)+ \Theta_{P}+\Theta_{T}+\Theta_{C},\qquad \tilde{b}-b = (b_{\ell}-b) + W_{P}.
\end{equation}
$\Theta_{P}$ and $W_{P}$ are called the principal perturbations, while $\Theta_{T}$ and $\Theta_{C}$ are time-average and space-average correctors, respectively. 

We now define these objects one by one. 
The principal perturbations are defined as follows:
\begin{gather*}
    \Theta_{P}(t,x) \coloneqq  \bar{r}^{-d/q'}L^{-1}\overunderset{2d}{j=1}{\sum}\tau_{j}(t)\theta_{j}(t,x)=\sigma^{-1}\overunderset{2d}{j=1}{\sum}\left(\frac{1}{L|\xi_{j}|}\int_{t}^{t+\tau_{j}(t)}R_{\ell}^{j}(s,x_{j}(s))|\dot{x}_{j}(s)|\,ds+\eta\right)\theta_{j}(t,x),\\
    W_{P}(t,x) \coloneqq  \sigma\overunderset{2d}{j=1}{\sum}w_{j}(t,x).
    \label{def-principal-perturbations}
\end{gather*}
Notice that $\Theta_{P}\ge 0$ and $W_{P}$ is divergence-free at each time. By point (ii) in Lemma \ref{lemma-buildingblocks}, the building blocks relative to different components have disjoint supports, so that the product of $\Theta_{P}$ and $W_{P}$ can be rewritten as
\begin{align*}
    \Theta_{P}(t,x)W_{P}(t,x)&=\bar{r}^{-d/q'}L^{-1}\sigma\overunderset{2d}{j=1}{\sum}\tau_{j}(t)\theta_{j}(t,x)w_{j}(t,x)\\
    &=\overunderset{2d}{j=1}{\sum}\left(\frac{1}{L|\xi_{j}|}\int_{t}^{t+\tau_{j}(t)}R_{\ell}^{j}(s,x_{j}(s))|\dot{x}_{j}(s)|\,ds+\eta\right)\theta_{j}(t,x)w_{j}(t,x).
\end{align*}

We now address the term $\Theta_{T}$. In order to motivate the forthcoming constructions, we anticipate that the role of $\partial_{t}\Theta_{T}$ is to replace the leading order term coming from the nonlinear interaction $\partial_{t}\Theta_{P}+\diver (\Theta_{P}W_{P})$ with the corresponding time average, the latter being in turn responsible for the main error cancellation (see also the next paragraph for a clearer global picture).  
We introduce the map $\psi_{j}:\mathbb{R}\times \mathbb{T}^{d}\rightarrow \mathbb{R}$, defined as
 \begin{align}
    \label{def psij}
        \psi_{j}(t,x) \coloneqq \partial_{\xi_{j}} R_{\ell}^{j}(t,x_{j}(t))D_{j}(t,x),
    \end{align}
    and also $\langle \psi_{j}\rangle:\mathbb{R}\times \mathbb{T}^{d}\to \mathbb{R}$, its time average computed over a period $[t,t+\tau_{j}(t)]$  with density $(L|\xi_{j}|)^{-1}|\dot{x}_{j}(s)|ds$:
    \begin{equation}
        \label{def-timeaverage-psij}
        \langle \psi_{j}\rangle(t,x) \coloneqq  \frac{1}{L|\xi_{j}|}\int_{t}^{t+\tau_{j}(t)}\psi_{j}(s,x)|\dot{x}_{j}(s)|\,ds.
    \end{equation}
    The time average corrector is then defined as
    \begin{equation}\label{def-thetaT}
    \Theta_{T}(t,x) \coloneqq \overunderset{2d}{j=1}{\sum}\int_{0}^{t}\Big(\langle \psi_{j}\rangle(s,x)-\psi_{j}(s,x) (L|\xi_{j}|)^{-1}\tau_{j}(s)|\dot{x}_{j}(s)|\Big)ds.
\end{equation}

Finally, the space average corrector $\Theta_{C}$, which is meant to impose the zero average condition on the whole density perturbation at each time, is given by
\begin{equation}\label{def-thetaC}
    \Theta_{C}(t) \coloneqq -\int_{\mathbb{T}^{d}}\Big(\Theta_{P}(t,x)+\Theta_{T}(t,x)\Big)dx.
\end{equation}
\paragraph{The new error.} We close this section by giving the definition of the new error $\tilde{R}$.
We recall that the new triplet $(\tilde{\rho},\tilde{b}, \tilde{R})$ solves the equation $-\diver\tilde{R}=\partial_{t}\tilde{\rho}+\diver(\tilde{\rho}\tilde{b})$. Expanding the right-hand side using the definition of the perturbations in \eqref{formulaperturbations} as well as the fact that $\big(\rho_{\ell}, b_{\ell},R_{\ell}+(\rho b)_{\ell}-\rho_{\ell}b_{\ell}\big)$ solves \eqref{CDE}, we deduce the necessary condition
\begin{align*}
    -\diver \tilde{R}&=-\diver \big((\rho b)_{\ell}-\rho_{\ell}b_{\ell}\big)\\
    &\quad\,-\diver R_{\ell}+\partial_{t}\Theta_{P}+\diver(\Theta_{P}W_{P})+\partial_{t}\Theta_{T}+\partial_{t}\Theta_{C}\\
    &\quad\,+\diver(\rho_{\ell}W_{P}+\Theta_{P}b_{\ell}+\Theta_{T}b_{\ell} + {\Theta_T W_P}).
\end{align*}
Notice that the term $\diver\big(\Theta_{C}(b_{\ell}+W_{P})\big)$ vanishes because $b_{\ell}+W_{P}$ is divergence-free and $\Theta_{C}$ is constant in space at each time. While the terms in the first and third rows are already small, the second line is where a nontrivial cancellation happens, up to some extra error terms. 

The first error term is the one coming from the preliminary regularization: 
\begin{equation}\label{def-tildeR1}
    \tilde{R}_{(1)} \coloneqq  (\rho b)_{\ell}-\rho_{\ell}b_{\ell}.
\end{equation}
We are then going to expand separately $-\diver R_{\ell}$, $\partial_{t}\Theta_{P}+\diver(\Theta_{P}W_{P})$, $\partial_{t}\Theta_{T}$ and $\partial_{t}\Theta_{C}$, and after that, add the resulting expressions together. For the sake of readability we drop the dependence on the variables. First of all, we have
\begin{equation*}
    -\diver R_{\ell}=-\overunderset{2d}{j=1}{\sum}\partial_{\xi_{j}} R_{\ell}^{j}.
\end{equation*}
Next, using the identity \eqref{eq-cancellation-buildingblocks} and the definition of $\psi_{j}$ from \eqref{def psij}, we derive 
\begin{align*}
    \partial_{t}\Theta_{P}+\diver(\Theta_{P}W_{P})&=\overunderset{2d}{j=1}{\sum}\bar{r}^{-d/q'}L^{-1}\tau_{j}\Big(\partial_{t}\theta_{j}+\sigma\diver(\theta_{j}w_{j})\Big)+\overunderset{2d}{j=1}{\sum}\bar{r}^{-d/q'}L^{-1}\partial_{t}\tau_{j}\theta_{j}\\
    &=\overunderset{2d}{j=1}{\sum}(L|\xi_{j}|)^{-1}\tau_{j}|\dot{x}_{j}|\psi_{j}+\overunderset{2d}{j=1}{\sum}\bar{r}^{-d/q'}L^{-1}\partial_{t}\tau_{j}\theta_{j}+\overunderset{2d}{j=1}{\sum}L^{-1}\tau_{j}\partial_{t}R_{\ell}^{j}D_{j}\\
    &\quad\,-\diver\left(\overunderset{2d}{j=1}{\sum}\frac{q'}{d}L^{-1}\tau_{j}(\partial_{t}R_{\ell}^{j}+\nabla R_{\ell}^{j}\cdot \dot{x}_{j})D_{j} (x-x_{j})\right).
\end{align*}
    Also, by the definition of $\Theta_{T}$ in \eqref{def-thetaT}, we get
    \begin{equation*}
        \partial_{t}\Theta_{T}=\overunderset{2d}{j=1}{\sum}\Big(\langle \psi_{j}\rangle-(L|\xi_{j}|)^{-1}\tau_{j}|\dot{x}_{j}|\psi_{j}\Big),
    \end{equation*}
    while, from the definition of $\Theta_{C}$ in \eqref{def-thetaC}, we obtain
    \begin{align*}
        \partial_{t}\Theta_{C}=-\int_{\mathbb{T}^{d}}\big(\partial_{t}\Theta_{P}+\partial_{t}\Theta_{T}\big)
        =-\int_{\mathbb{T}^{d}}\overunderset{2d}{j=1}{\sum}\bar{r}^{-d/q'}L^{-1}\partial_{t}\tau_{j}\theta_{j}-\int_{\mathbb{T}^{d}}\overunderset{2d}{j=1}{\sum}L^{-1}\tau_{j}\partial_{t}R_{\ell}^{j}D_{j}-\int_{\mathbb{T}^{d}}\overunderset{2d}{j=1}{\sum}\langle\psi_{j}\rangle.
    \end{align*}
    Collecting the expressions obtained so far we get:
    \begin{align*}
        -\diver R_{\ell}+\partial_{t}\Theta_{P}+\diver(\Theta_{P}W_{P})+\partial_{t}\Theta_{T}+\partial_{t}\Theta_{C}&=\overunderset{2d}{j=1}{\sum}\left(\langle \psi_{j}\rangle-\int_{\mathbb{T}^{d}}\langle \psi_{j}\rangle-\partial_{\xi_{j}} R_{\ell}^{j}\right)\\
        &\quad\,-\diver\left(\tilde{R}_{(2)}+\tilde{R}_{(3)}+\tilde{R}_{(4)}\right),
    \end{align*}
    where we have introduced the following error terms:
    \begin{align}
        \tilde{R}_{(2)}& \coloneqq -\nabla \Delta^{-1}\left(\overunderset{2d}{j=1}{\sum}\bar{r}^{-d/q'}L^{-1}\partial_{t}\tau_{j}\theta_{j}-\int_{\mathbb{T}^{d}}\overunderset{2d}{j=1}{\sum}\bar{r}^{-d/q'}L^{-1}\partial_{t}\tau_{j}\theta_{j}\right),\label{def-tildeR2}\\
        \tilde{R}_{(3)}& \coloneqq  -\nabla \Delta^{-1}\left(\overunderset{2d}{j=1}{\sum}L^{-1}\tau_{j}\partial_{t}R_{\ell}^{j}D_{j}-\int_{\mathbb{T}^{d}}\overunderset{2d}{j=1}{\sum}L^{-1}\tau_{j}\partial_{t}R_{\ell}^{j}D_{j}\right),\label{def-tildeR3}\\
        \tilde{R}_{(4)}& \coloneqq  \overunderset{2d}{j=1}{\sum}\frac{q'}{d}L^{-1}\tau_{j}(\partial_{t}R_{\ell}^{j}+\nabla R_{\ell}^{j}\cdot \dot{x}_{j})D_{j} (x-x_{j})\label{def-tildeR4}.
    \end{align}
    Let us examine the term $\langle \psi_{j}\rangle-\int_{\mathbb{T}^{d}}\langle \psi_{j}\rangle-\partial_{\xi_{j}} R_{\ell}^{j}$. 
    We introduce an auxiliary quantity, suitably close to $\langle \psi_{j}\rangle$:
    \begin{equation*}
        \varphi_{j}(t,x) \coloneqq  \partial_{\xi_{j}}R_{\ell}^{j}(t,x)\frac{1}{L|\xi_{j}|}\int_{t}^{t+\tau_{j}(t)}D_{j}(s,x)|\dot{x}_{j}(s)|\,ds
    \end{equation*}
    and we rewrite
        \begin{align*}
        \langle \psi_{j}\rangle
        -\partial_{\xi_{j}}R_{\ell}^{j}&= \partial_{\xi_{j}}R_{\ell}^{j}\left(\frac{1}{L|\xi_{j}|}\int_{t}^{t+\tau_{j}(t)}D_{j}|\dot{x}_{j}|\,ds-1\right) +  \langle \psi_{j}\rangle-\varphi_{j}. 
    \end{align*}
    From this last expression two more error terms $\tilde{R}_{(5)}$ and $\tilde{R}_{(6)}$ arise, satisfying $$-\diver(\tilde{R}_{(5)}+\tilde{R}_{(6)})= \overunderset{2d}{j=1}{\sum}\left(\langle \psi_{j}\rangle-\int_{\mathbb{T}^{d}}\langle \psi_{j}\rangle-\partial_{\xi_{j}} R_{\ell}^{j}\right).$$ By Lemma \ref{lemimpantidivcrawling}, there exists a field $v_{j}\in C^{\infty}(\mathbb{R}\times\mathbb{T}^{d};\mathbb{R}^{d})$ such that
    \begin{equation}\label{propertiesvj}
        \diver v_{j}(t, x)=\frac{1}{L|\xi_{j}|}\int_{t}^{t+ \tau_j(t)}D_j(s, x)|\dot x_j(s)|\,ds-1,\qquad \lVert v_{j}\rVert_{C_{t}L_{x}^{1}}\le c_{d}\lambda^{-1}.
    \end{equation} 
    Then, recalling the definition of the anti-divergence operator $\mathcal{R}$ from \eqref{def-smart-anti-divergence}, we set 
    \begin{align}
        \tilde{R}_{(5)}& \coloneqq -\overunderset{2d}{j=1}{\sum}\mathcal{R}\Big(\partial_{\xi_{j}}R_{\ell}^{j},v_{j}\Big),\label{def-tildeR5}\\
        \tilde{R}_{(6)}& \coloneqq -\nabla \Delta^{-1}\overunderset{2d}{j=1}{\sum}\left(\langle \psi_{j}\rangle-\varphi_{j} -\int_{\mathbb{T}^{d}}(\langle \psi_{j}\rangle-\varphi_{j})\right).\label{def-tildeR6}
    \end{align}
    Finally, we introduce the last error term:
    \begin{equation}\label{def-tildeR7}
        \tilde{R}_{(7)} \coloneqq  -(\rho_{\ell}W_{P}+\Theta_{P}b_{\ell}+\Theta_{T}b_{\ell}+\Theta_{T}W_{P}).
    \end{equation}
    
    In conclusion, taking into account all the computations made in this paragraph it is not difficult to see that by defining the new error $\tilde{R}$ as
    \begin{equation}\label{def-tildeR}
    \tilde{R}=\tilde{R}_{(1)}+\tilde{R}_{(2)}+\tilde{R}_{(3)}+\tilde{R}_{(4)}+\tilde{R}_{(5)}+\tilde{R}_{(6)}+\tilde{R}_{(7)},
    \end{equation}
    with $\tilde{R}_{(1)},\dots,\tilde{R}_{(7)}$ defined respectively in \eqref{def-tildeR1}, \eqref{def-tildeR2}, \eqref{def-tildeR3}, \eqref{def-tildeR4}, \eqref{def-tildeR5}, \eqref{def-tildeR6} and \eqref{def-tildeR7}, the new triplet $(\tilde{\rho}, \tilde{b},\tilde{R})$ is indeed a solution of \eqref{CDE}.
    
\section{Estimates and proof of the perturbation theorem}\label{sec_estimates}
In this final section we derive all the necessary estimates on the new approximate solution $(\tilde{\rho},\tilde{b},\tilde{R})$ defined in Section \ref{sec_constructions} and eventually give a proof of Theorem \ref{thm_pert}. Recall that we are working under the condition $$\alpha \coloneqq 1+\frac{1}{d}-\frac{1}{p}-\frac{1}{q}\ge 0,$$ so that the Sobolev embedding $W^{1,p}\hookrightarrow L^{q'}$ holds true. 

Our goal is to find $\delta(d,q)>0$ for which the following holds under the hypothesis $\alpha <\delta(d,q)$: there exists $\bar{\eta}>0$ small enough and constants $a_{1}>0, a_{2}>1, a_{3}>0, a_{4}>1$ such that, for every $T\in [1,2]$, if $(\rho, b, R)$ solves \eqref{CDE} and satisfies the estimates
\begin{subequations}
    \begin{gather}
       \eta \coloneqq  \underset{t\in [0,T]}{\max}\lVert R(t,\cdot)\rVert_{L^{1}} \le \bar{\eta},\label{assumption-error}\\
       \underset{t\in [0,T]}{\max}\Big(\lVert \rho(t,\cdot)\rVert_{L^{q}}+\lVert b(t,\cdot)\rVert_{W^{1,p}}\Big)\le 1-2\eta^{a_{1}},\label{assumption-density&field}\\
       \underset{t\in [0,T]}{\max}\Big(\lVert \nabla_{t,x}\rho(t,\cdot)\rVert_{L^{q}}+\lVert \nabla_{t,x}b(t,\cdot)\rVert_{W^{1,p}}\Big)\le \eta^{-a_{2}},\label{assumption-higher-order-derivatives-density&field}
    \end{gather}
\end{subequations}
then, the new triplet $(\tilde{\rho}, \tilde{b}, \tilde{R})$ defined in Section \ref{sec_constructions}, which still solves \eqref{CDE}, satisfies the estimates
\begin{subequations}
    \begin{gather}
        \underset{t\in [0,T-\eta^{a_{3}}]}{\max}\lVert \tilde{R}(t,\cdot)\rVert_{L^{1}}\le \eta^{a_{4}},\label{conclusion-error}\\
        \underset{t\in [0,T-\eta^{a_{3}}]}{\max}\Big(\lVert \tilde{\rho}(t,\cdot)\rVert_{L^{q}}+\lVert \tilde{b}(t,\cdot)\rVert_{W^{1,p}}\Big)\le 1-2\eta^{a_{1}a_{4}},\label{conclusion-density&field}\\
        \underset{t\in [0,T-\eta^{a_{3}}]}{\max}\Big(\lVert \nabla_{t,x}\tilde{\rho}(t,\cdot)\rVert_{L^{q}}+\lVert \nabla_{t,x}\tilde{b}(t,\cdot)\rVert_{W^{1,p}}\Big)\le \eta^{-a_{2}a_{4}},\label{conclusion-higher-order-derivatives-density&field}\\
        \underset{t\in [0,T-\eta^{a_{3}}]}{\max}\Big(\lVert \tilde{\rho}(t,\cdot)-\rho(t,\cdot)\rVert_{L^{q}}+\lVert \tilde{b}(t,\cdot)-b(t,\cdot)\rVert_{W^{1,p}}\Big)\le \eta^{a_{1}}. \label{conclusion-perturbation-distance}
    \end{gather}
\end{subequations}
     
To this end, we will choose each of the auxiliary parameters introduced in Section \ref{sec_constructions} as an appropriate power of $\eta$. Namely, with $\bar{r}$ denoting the typical concentration scale, $\lambda\in \mathbb{N}$ the periodicity scale, $\ell$ the regularization scale, and $\sigma$ the velocity tuning, we take   
\begin{align}\label{eq-choice-shape-parameters}
\bar{r} = \eta^\mu, \qquad \lambda = \lceil 1/\eta^{\kappa} \rceil, \qquad \ell = \eta^\beta, \qquad \sigma = \eta^\gamma, 
\end{align}
where $\mu, \kappa, \beta, \gamma>0$ are parameters to be chosen. As we shall see, to close the iteration scheme, these parameters, along with the constants $a_1, a_2, a_3$ and $a_4$ are required to satisfy a few simple inequalities, up to choosing $\bar{\eta}$ sufficiently small. For the reader's convenience, we will specify one feasible choice at the outset that meets all the necessary requirements: 
\begin{align}\label{eq-choice-parameters}
\beta = 81 q^\prime d^2 (d+q+q'), \quad \kappa = 3 \beta d, \quad \gamma = \frac{1}{2}, \quad \mu = 4 \beta d q^\prime (d+q+q'), \quad a_1 = \frac{1}{4}, \quad 3 a_2 = 2a_3 = 3 a_4 = \beta. 
\end{align}
With these choices, we then take $\delta(d,q)$ to be $\frac{1}{6} \frac{1}{\mu d + q^\prime}$, which rewrites as
\begin{align}\label{formula-gap}
\delta(d,q)= \frac{1}{1944 d^4 q'^{2}(d+q+q')^2 + 6 q^\prime}.
\end{align}
\begin{rmk}
    Recall the structural constraints \eqref{eq-constraint-parameters1} and \eqref{eq-constraint-parameters2} we need to impose on the parameters in order to be allowed to use all the constructions and bounds from Section \ref{sec_constructions} (see Remark \ref{rmk-constraints-parameters}): 
        \begin{gather*}
             \ell^{-d-1}(\lambda^{-1}+\bar{r}^{d/q'}\lambda^{d-1}\ell^{-d}\sigma^{-1}\eta)\le 1,\label{constraint-parameters1}\\
             \bar{r}\ell^{-q'}\eta^{q'/d}\le c_{d,q}\lambda^{-\frac{2d-2}{d-2}}.\label{constraint-parameters2}
        \end{gather*}
These conditions can be recast as follows:
\begin{align}
\eta^{\kappa - \beta (d+1)} + \eta^{1 + \mu \frac{d}{q^\prime} - \kappa(d-1) - \beta (2d + 1) - \gamma} \leq 1, \qquad \qquad \eta^{\mu + \frac{q^\prime}{d} - \beta q^\prime - \kappa \frac{2d - 2}{d - 2}} \leq c_{d, q}. \nonumber
\end{align}\fr
\end{rmk}

\begin{rmk}    
Observe that all the constructions and estimates from Section \ref{sec_constructions}, when considered for times in $[0,T-\eta^{a_{3}}]$, only depend on the restriction of $(\rho, b, R)$ to $[0,T]\times \mathbb{T}^{d}$, provided that $|\tau_{j}(t)|, \ell \le \eta^{a_{3}}/2$. Thanks to \eqref{bounds-tauj}, to have that, we just need to impose the conditions
        \begin{gather*}
            c_{d,q} \bar{r}^{d/q'}L\sigma^{-1}\eta\le \eta^{a_{3}}/2,\label{constraint-localization-time1} \nonumber \\
        \ell \le \eta^{a_{3}}/2\label{constraint-localization-time2}, \nonumber
        \end{gather*} 
which can also be rewritten as
\begin{align}
\eta^{1 + \mu \frac{d}{q^\prime} - \kappa(d-1)  - \gamma - a_3} \leq c_{d,q}, \qquad \qquad \eta^{\beta - a_3} \leq 1/2. \nonumber
\end{align}
    In particular, in the following pages, in deriving each estimate, uniform norms in time will be considered in the interval $[0,T-\eta^{a_{3}}]$ for the left-hand side and in the interval $[0,T]$ for the right-hand side.\fr
\end{rmk}

\paragraph{Estimates on the perturbations.} In this paragraph we derive the required estimates on the perturbations $\tilde{\rho}-\rho$ and $\tilde{b}-b$ as defined in \eqref{formulaperturbations}.

\medskip
\noindent ($\lVert \rho_{\ell}-\rho\rVert_{C_{t}L_{x}^{q}}$). From \eqref{bound-correction-convolution-density&field} and \eqref{assumption-higher-order-derivatives-density&field}, we get
\begin{equation}\label{eq-bound-rhoell-rho}
    \begin{aligned}
    \lVert \rho_{\ell}-\rho\rVert_{C_{t}L_{x}^{q}}&\le c_{d,q}\ell \lVert \nabla_{t,x}\rho\rVert_{C_{t}L_{x}^{q}}\\
    &\leq c_{d, q} \eta^{\beta - a_2}< \eta^{a_1}/6.
\end{aligned}
\end{equation}

\medskip
\noindent ($\lVert \Theta_{P}\rVert_{C_{t}L_{x}^{q}}$). Using the scaling of $\theta_{j}$ in \eqref{scalingbbdensity} and then the upper bound on $\tau_{j}$ in \eqref{bounds-tauj}, we obtain
    \begin{equation}\label{bdd thetaP Lq}
    \begin{split}
        \lVert \Theta_{P}\rVert_{C_{t}L_{x}^{q}} &\le c_{d,q}\bar{r}^{-d/q'}L^{-1}\sup_{t, j}\tau_{j}\\
        &\le c_{d,q}\sigma^{-1}\eta = c_{d,q} \eta^{1 - \gamma} < \eta^{a_1}/6.
    \end{split} 
    \end{equation}
Similarly, taking the $L^{1}$-norm one gets
\begin{equation}\label{bdd thetaP L1}
\begin{split}
     \lVert \Theta_{P}\rVert_{C_{t}L_{x}^{1}} &\le c_{d,q}\bar{r}^{-d/q'}L^{-1}r_{\max}^{d/q'}\sup_{t, j}\tau_{j}\\
      &\le c_{d,q}\bar{r}^{d/q'}\ell^{-d}\sigma^{-1}\eta^{2}.
\end{split}
\end{equation}
\medskip
\noindent ($\lVert \Theta_{T}\rVert_{C_{t}L_{x}^{q}}$). 
Calling $g_{j}(t,x) \coloneqq  \partial_{\xi_{j}}R_{\ell}^{j}(t,x_{j}(t))D_{j}(t,x)|\dot{x}_{j}(t)|/(L|\xi_{j}|)$, we see that $\Theta_{T}$ can be rewritten as 
\begin{equation*}
    \Theta_{T}(t,x)= \overunderset{2d}{j=1}{\sum}\left(\int_{0}^{t}\int_{s}^{s+\tau_j(s)}g_{j}(r,x)\,dr\,ds -\int_{0}^{t} \tau_j(r)g_{j}(r,x)\,dr \right).
\end{equation*}
Now we use Fubini's theorem to show that some cancellation occurs in the previous expression.
By \eqref{upperbound-dettauj} and \eqref{eq-constraint-parameters1}, up to decreasing $\bar\eta$ we have $\sup_t |\partial_t\tau_j(t)|\le 1/2$. Hence $t\mapsto t+\tau_j(t)$ has derivative bounded below by $1/2$ and is a smooth diffeomorphism of $\mathbb{R}$ onto itself.
Let $\zeta_{j}$ be the inverse function of $\id+\tau_{j}$, so that $\zeta_{j}(s)+\tau_{j}(\zeta_{j}(s))=s$. We have
\begin{align*}
    \int_{0}^{t}\int_{s}^{s+\tau_{j}(s)}g_{j}(r,x)\,dr\,ds&=\int_{\mathbb{R}^{2}}\mathbbm{1}_{[0,t]}(s)\mathbbm{1}_{[s,s+\tau_{j}(s)]}(r)g_{j}(r,x)\,dr\,ds\\
    &=\int_{\mathbb{R}^{2}}\mathbbm{1}_{[0,t]\cap [\zeta_{j}(r),r]}(s)g_{j}(r,x)\,dr\,ds\\
    &= \int_{0}^{\tau_{j}(0)}rg_{j}(r,x)\,dr+\int_{t}^{t+\tau_{j}(t)}(t-\zeta_{j}(r))g_{j}(r,x)\,dr+\int_{\tau_{j}(0)}^{t}\tau_{j}(\zeta_{j}(r))g_{j}(r,x)\,dr.
\end{align*}
Hence,
\begin{align*}
    \Theta_{T}(t,x)&=\overunderset{2d}{j=1}{\sum}\Bigg(\int_{0}^{\tau_{j}(0)}(r-\tau_{j}(r))g_{j}(r,x)\,dr+\int_{t}^{t+\tau_{j}(t)}(t-\zeta_{j}(r))g_{j}(r,x)\,dr\\
    &\quad\,+\int_{\tau_{j}(0)}^{t}\big(\tau_{j}(\zeta_{j}(r))-\tau_{j}(r)\big)g_{j}(r,x)\,dr\Bigg).
\end{align*}
Now we use the mean value theorem and then we take the $L^{q}$-norm to get
\begin{equation}\label{bdd thetaT Lq}
    \begin{split}
         \lVert\Theta_{T}\rVert_{C_{t}L^{q}_{x}}&\le c_{d,q}\left(\sup_{t,j}\tau_{j}+\sup_{t,j}|\partial_{t}\tau_{j}|\right)\lVert \nabla R_{\ell}\rVert_{C_{t,x}}\sup_{t,j}\Big\|\frac{1}{L|\xi|}\int_{t}^{t+\tau_{j}(t)}D_{j}(r,\cdot)|\dot{x}_{j}(r)|\,dr\Big\|_{L^{q}}\\
       &\le c_{d,q}\bar{r}^{1/q'}\lambda^{(d-1)/q}\ell^{-d/(q-1)-2d-3}\sigma^{-1}\eta^{1+1/d} \\
    &= c_{d,q} \eta^{\frac{\mu}{q^\prime} - \kappa \frac{(d-1)}{q} - \beta \left(\frac{d}{q-1} + 2d + 3\right) - \gamma + 1 + \frac{1}{d}} < \eta^{a_1}/6.
    \end{split}
\end{equation}
In the second inequality we used Lemma \ref{lemnormtraceleftroundtorus} to bound the $L^{q}$-norm of the time average of $D_{j}(t,x)$ in a period. Moreover, we employed the upper bounds on $\tau_{j}$ and $|\partial_{t}\tau_{j}|$ from \eqref{bounds-tauj} and \eqref{upperbound-dettauj}, together with the uniform estimate on $\nabla R_{\ell}$ in \eqref{boundRconvolution}. 

\noindent Similarly, taking the $L^{1}$-norm, using the fact that $\lVert D_{j}(t,\cdot)\rVert_{L^{1}}=1$, one gets
\begin{equation}\label{bdd thetaT L1}
    \begin{split}
         \lVert\Theta_{T}\rVert_{C_{t}L^{1}_{x}}&\le c_{d,q}\left(\sup_{t,j}\tau_{j}+\sup_{t,j}|\partial_{t}\tau_{j}|\right)\lVert \nabla R_{\ell}\rVert_{C_{t,x}}\\
       &\le c_{d,q}\bar{r}^{d/q'}\lambda^{d-1}\ell^{-2d-2}\sigma^{-1}\eta^{2}.
    \end{split}
\end{equation}

\medskip
\noindent ($\lVert \Theta_{C}\rVert_{C_{t,x}}$). We have $|\Theta_{C}(t)|\le \lVert \Theta_{P}(t,\cdot)\rVert_{L^{1}}+ \lVert \Theta_{T}(t,\cdot)\rVert_{L^{1}}$. Therefore, by \eqref{bdd thetaP L1} and \eqref{bdd thetaT L1},
\begin{equation}\label{eq-bound-Thetac}
    \begin{aligned}
        \lVert \Theta_{C}\rVert_{C_{t,x}}&\le \lVert \Theta_{P}\rVert_{C_{t}L_{x}^{1}}+ \lVert \Theta_{T}\rVert_{C_{t}L_{x}^{1}} \\ 
        &\le c_{d,q}\bar{r}^{d/q'}\lambda^{d-1}\ell^{-2d-2}\sigma^{-1}\eta^{2}\\
        & \leq c_{d,q} \eta^{\mu\frac{d}{q^\prime} - \kappa (d-1) - 2 \beta \left(d + 1\right) - \gamma + 2} < \eta^{a_1}/6.
    \end{aligned}
\end{equation}

\medskip
\noindent ($\lVert b_{\ell}-b\rVert_{C_{t}W_{x}^{1,p}}$). From \eqref{bound-correction-convolution-density&field} and \eqref{assumption-higher-order-derivatives-density&field}, we get
\begin{equation}\label{eq-bound-bell-b}
    \begin{aligned}
    \lVert b_{\ell}-b\rVert_{C_{t}W_{x}^{1,p}}&\le c_{d,p}\ell \lVert \nabla_{t,x}b\rVert_{C_{t}W_{x}^{1,p}}\\
    &\leq c_{d,p} \eta^{\beta - a_2} < \eta^{a_1}/6.
\end{aligned}
\end{equation}

\medskip
\noindent ($\lVert W_{P}\rVert_{C_{t}W_{x}^{1,p}}$). Using the scaling of $w_{j}$ in \eqref{scalingbbfield}, together with the formula for $r_{\min}$ in \eqref{defn_rminrmax}, we get
\begin{equation}\label{bdd WP W1p}
    \begin{split}
         \lVert W_{P}\rVert_{C_{t}W_{x}^{1,p}}&\le c_{d,p}\sigma r_{\min}^{-d\alpha}\\
    &=c_{d,p}\bar{r}^{-d\alpha}\sigma\eta^{-q'\alpha} \leq c_{d, q} \eta^{-\mu d \alpha +\gamma - q^\prime \alpha} < \eta^{a_1}/6.
    \end{split}
\end{equation}
Similarly, taking the $L^{1}$-norm one gets
\begin{equation}\label{bdd WP L1}
    \begin{split}
         \lVert W_{P}\rVert_{C_{t}L^{1}_{x}}&\le c_{d,q}\sigma r_{\max}^{d/q}\\
    &\le c_{d,q}\bar{r}^{d/q}\ell^{-d/(q-1)}\sigma \eta^{1/(q-1)}.
    \end{split}
\end{equation}
\paragraph{Estimates on higher order derivatives of density and field.} In this paragraph we derive the estimates on the higher order derivatives of the new density and field $\tilde{\rho}$ and $\tilde{b}$.

\medskip
\noindent ($\lVert \nabla_{t,x}\tilde{\rho}\rVert_{C_{t}L_{x}^{q}}$). Recall that $\tilde{\rho}= \rho_{\ell}+\Theta_{P}+\Theta_{T}+\Theta_{C}$. By Young's inequality and \eqref{assumption-higher-order-derivatives-density&field}, we have
\begin{equation}\label{eq-bound-der-rhoell}
    \begin{aligned}
    \lVert \nabla_{t,x}\rho_{\ell}\rVert_{C_{t}L_{x}^{q}} &\leq \lVert \nabla_{t,x}\rho\rVert_{C_{t}L_{x}^{q}}\\
    &\leq \eta^{-a_2} < \eta^{-a_2 a_4}/8.
\end{aligned}
\end{equation}
Let us consider the derivatives of $\Theta_{P}$. We have
\begin{equation*}
        \partial_{t}\Theta_{P}=\bar{r}^{-d/q'}L^{-1}\overunderset{2d}{j=1}{\sum}\Big(\partial_{t}\tau_{j}\theta_{j}+\tau_{j}\partial_{t}\theta_{j}\Big),\qquad \nabla \Theta_{P}= \bar{r}^{-d/q'}L^{-1}\overunderset{2d}{j=1}{\sum}\tau_{j}\nabla \theta_{j}.
\end{equation*}
We first use the estimates for the $L^{q}$-norm of $\theta_{j}, \partial_{t} \theta_{j}$ and $\nabla \theta_{j}$ from \eqref{scalingbbdensity} and \eqref{eq-bound-det-thetaj}. Then, the upper bounds on $\tau_{j}$ and $|\partial_{t}\tau_{j}|$ from \eqref{bounds-tauj} and \eqref{upperbound-dettauj}, together with the bounds on $r_{j},\partial_{t}r_{j}, \nabla r_{j}$ and $|\dot{x}_{j}|$ from 
\eqref{defn_rminrmax}, \eqref{eq-bounds-derivatives-rj} and \eqref{boundsspeed}. We thus obtain
\begin{equation}\label{eq-bound-der-Thetap}
    \begin{aligned}
        \lVert \partial_{t}\Theta_{P} \rVert_{C_{t}L_{x}^{q}}+\lVert \nabla \Theta_{P}\rVert_{C_{t}L_{x}^{q}}&\le c_{d,q}\bar{r}^{-d/q'}L^{-1}\sup_{t,j}\Big(|\partial_{t}\tau_{j}|+\tau_{j}r_{j}^{-1}(1+|\partial_{t}r_{j}|+|\nabla r_{j}|)(1+|\dot{x}_{j}|) \Big)\\
        &\le c_{d,q}\left(\ell^{-d-1}\sigma^{-1}\eta+\bar{r}^{-d/q'}\ell^{-d-q'-1}\right) \\
        & \leq c_{d, q} \left(\eta^{1 - \gamma - \beta(d+1)} + \eta^{-\mu \frac{d}{q^\prime} - \beta (d + q^\prime + 1)}\right)
        < \eta^{-a_2 a_4}/8.
\end{aligned}
\end{equation}
Now we pass to the derivatives of $\Theta_{T}$. By the definition of $\Theta_{T}$ in \eqref{def-thetaT}, we have
\begin{gather*}
    \partial_{t}\Theta_{T}=\overunderset{2d}{j=1}{\sum}\frac{1}{L|\xi_{j}|}\int_{t}^{t+\tau_{j}(t)}\Big(\psi_{j}(s,x)|\dot{x}_{j}(s)|-\psi_{j}(t,x)|\dot{x}_{j}(t)|\Big)ds,\\\nabla \Theta_{T}= \overunderset{2d}{j=1}{\sum}\frac{1}{L|\xi_{j}|}\int_{0}^{t}\int_{s}^{s+\tau_{j}(s)}\Big(\nabla\psi_{j}(r,x)|\dot{x}_{j}(r)|-\nabla \psi_{j}(s,x)|\dot{x}_{j}(s)|\Big)dr\,ds.
\end{gather*}
To bound $\partial_{t}\Theta_{T}$, first observe that $D_{j}(t,\cdot)$ defined in \eqref{def-Dj} can be bounded in $L^{q}$ by
\begin{align*}
\lVert D_j(t, \cdot) \rVert_{L^{q}} \leq c_{d, q} r_{j}(t,x_{j}(t))^{-d/q'}.    
\end{align*}
Then,
\begin{equation}\label{eq-bound-tder-Thetat}
    \begin{aligned}
    \lVert \partial_{t}\Theta_{T}\rVert_{C_{t}L_{x}^{q}}&\le c_{d,q}L^{-1}\sup_{t,j}\Big(\tau_{j}|\dot{x}_{j}|r_{j}^{-d/q'}|\partial_{\xi_{j}}R_{\ell}^{j}|\Big)\\
    &\le c_{d,q}\bar{r}^{-d/q'}\ell^{-d-1}\\
       & \leq c_{d, q} \eta^{- \mu \frac{d}{q^\prime} - \beta (d+1)} < \eta^{-a_2 a_4}/8.
\end{aligned}
\end{equation}
In the second inequality we used the bounds on $r_{j}, |\dot{x}_{j}|$ and $\tau_{j}$ from \eqref{defn_rminrmax}, \eqref{boundsspeed} and \eqref{bounds-tauj}, together with the uniform estimate on $\nabla R_{\ell}$ from \eqref{boundRconvolution}.

\noindent To estimate $\lVert\nabla \Theta_{T}\rVert_{C_{t}L_{x}^{q}}$ one may proceed exactly as in the derivation of the bound for $\lVert \Theta_{T}\rVert_{C_{t}L_{x}^{q}}$ from the previous paragraph, using $\nabla g_{j}(t,x)$ in place of $g_{j}(t,x)$, thus obtaining an additional factor $r_{\min}^{-1}$ in the bound, i.e.
\begin{equation}\label{eq-bound-xder-Thetat}
    \begin{split}
         \lVert\nabla\Theta_{T}\rVert_{C_{t}L^{q}_{x}}&\le c_{d,q}\left(\sup_{t,j}\tau_{j}+\sup_{t,j}|\partial_{t}\tau_{j}|\right)\lVert \nabla R_{\ell}\rVert_{C_{t,x}}\sup_{t,j}\Big\|\frac{1}{L|\xi|}\int_{t}^{t+\tau_{j}(t)}\nabla D_{j}(r,\cdot)|\dot{x}_{j}(r)|\,dr\Big\|_{L^{q}}\\
       &\le c_{d,q}\bar{r}^{-1/q}\lambda^{(d-1)/q}\ell^{-d/(q-1)-2d-3}\sigma^{-1}\eta^{1-1/d(q-1)} \\
    &= c_{d,q} \eta^{-\mu\frac{1}{q} - \kappa \frac{(d-1)}{q} - \beta \left(\frac{d}{q-1} + 2d + 3\right) - \gamma + 1 -\frac{1}{d(q-1)}} 
    < \eta^{-a_2 a_4}/8.
    \end{split}
\end{equation}
Finally, recall that $\Theta_{C}$ is a constant in space at each time, hence $\nabla \Theta_{C}\equiv 0$. Moreover 
\begin{equation}\label{eq-bound-tder-Thetac}
    \lVert \partial_{t}\Theta_{C}\rVert_{C_{t}L_{x}^{q}}\le \lVert \partial_{t}\Theta_{P}\rVert_{C_{t}L_{x}^{q}}+\lVert \partial_{t}\Theta_{T}\rVert_{C_{t}L_{x}^{q}}  < \eta^{-a_2 a_4}/4.
\end{equation}

\medskip
\noindent ($\lVert \nabla_{t,x} \tilde{b}\rVert_{C_{t}W_{x}^{1,p}}$). Recall that $\tilde{b}=b_{\ell}+W_{P}$. By Young's inequality and \eqref{assumption-higher-order-derivatives-density&field}, we have
\begin{equation}\label{eq-bound-der-bell}
     \begin{aligned}
     \lVert \nabla_{t,x} b_{\ell}\rVert_{C_{t}W_{x}^{1,p}}&\le \lVert \nabla_{t,x} b \rVert_{C_{t}W_{x}^{1,p}}\\
     &\leq \eta^{-a_2} < \eta^{-a_2 a_4}/8.
\end{aligned}
\end{equation}
Let us now address the derivatives of $W_{P}$. We have
\begin{equation*}
    \partial_{t}W_{P}=\sigma\overunderset{2d}{j=1}{\sum}\partial_{t}w_{j},\qquad \nabla W_{P}=\sigma\overunderset{2d}{j=1}{\sum}\nabla w_{j}.
\end{equation*}
We first use the estimates for the $W^{1,p}$-norm of $\partial_{t} w_{j}$ and $\nabla w_{j}$ from \eqref{eq-bound-det-wj} and \eqref{scalingbbfield}, respectively. Then, the bounds on $r_{j},\partial_{t}r_{j}, \nabla r_{j}$ and $|\dot{x}_{j}|$ from 
\eqref{defn_rminrmax}, \eqref{eq-bounds-derivatives-rj} and \eqref{boundsspeed}. We thus obtain
\begin{equation}\label{eq-bound-der-Wp}
    \begin{aligned}
        \lVert \partial_{t}W_{P} \rVert_{C_{t}W_{x}^{1,p}}+\lVert \nabla W_{P}\rVert_{C_{t}W_{x}^{1,p}} &\le c_{d,p}\sigma\sup_{t,j}\Big(r_{j}^{-1-d\alpha}(1+|\partial_{t}r_{j}|+|\nabla r_{j}|)(1+|\dot{x}_{j}|)\Big)\\
        &\le c_{d,p,q}\bar{r}^{-d\alpha-d/q'}\ell^{-d-q'-1}\sigma^{2}\eta^{-1-q'\alpha} \\
        & \leq c_{d, p, q} \eta^{-\mu(d\alpha + \frac{d}{q^\prime})  - \beta (d + q^\prime + 1) + 2 \gamma - 1-q'\alpha}
        < \eta^{-a_2 a_4}/8.
\end{aligned}
\end{equation}

\paragraph{Estimates on the new error.} In this paragraph we will give a bound on the $C_{t}L_{x}^{1}$-norm of each term from \eqref{def-tildeR} composing the new error $\tilde{R}$.

\medskip
\noindent ($\lVert \tilde{R}_{(1)}\rVert_{C_{t}L_{x}^{1}}$). From \eqref{bound-correction-convolution-product}, the Sobolev embedding $W^{1,p}\hookrightarrow L^{q'}$ and \eqref{assumption-higher-order-derivatives-density&field}, we get
\begin{equation}\label{eq-bound-tildeR1}
    \begin{aligned}
        \lVert \tilde{R}_{(1)}\rVert_{C_{t}L_{x}^{1}} &\le c_{d}\ell^{2}\lVert \nabla_{t,x}\rho\rVert_{C_{t}L_{x}^{q}}\lVert \nabla_{t,x}b\rVert_{C_{t}L_{x}^{q'}}\\
        & \le c_{d,p,q}\ell^{2}\lVert \nabla_{t,x}\rho\rVert_{C_{t}L_{x}^{q}}\lVert \nabla_{t,x}b\rVert_{C_{t}W_{x}^{1,p}}\\
         &\leq c_{d,p,q} \eta^{2 \beta - 2a_2} < \eta^{a_4}/7.
    \end{aligned}
\end{equation}

\medskip
\noindent ($\lVert \tilde{R}_{(2)}\rVert_{C_{t}L_{x}^{1}}$). We first use Lemma \ref{lemstandardantidiv} to bound the $L^{1}$-norm of the standard anti-divergence. Then we use the bound on the $L^{1}$-norm of $\theta_{j}$ in \eqref{scalingbbdensity} and the bound on $\partial_{t}\tau_{j}$ in \eqref{upperbound-dettauj} along with the upper bound on $r_{j}$ from \eqref{defn_rminrmax}. Thus, we derive
\begin{equation}\label{eq-bound-tildeR2}
    \begin{aligned}
        \lVert \tilde{R}_{(2)}\rVert_{C_{t}L_{x}^{1}}&\le c_{d}\Big\| \overunderset{2d}{j=1}{\sum} \bar{r}^{-d/q'}L^{-1}\partial_{t}\tau_{j}\theta_{j}\Big\|_{C_{t}L_{x}^{1}}\\
        &\le c_{d,q}\bar{r}^{-d/q'}L^{-1}\sup_{t,j}\Big(|\partial_{t}\tau_{j}| r_{j}^{d/q'}\Big) \\
        &\le c_{d,q}\bar{r}^{d/q'}\ell^{-2d-1}\sigma^{-1}\eta^{2}
        \leq c_{d,q} \eta^{\mu \frac{d}{q^\prime}- \beta(2d + 1) - \gamma  + 2} < \eta^{a_4}/7.
    \end{aligned}
\end{equation}

\medskip
\noindent ($\lVert \tilde{R}_{(3)}\rVert_{C_{t}L_{x}^{1}}$). We first use Lemma \ref{lemstandardantidiv}, together with $\lVert D_{j}(t,\cdot)\rVert_{L^{1}}=1$. Then we upper bound $\tau_{j}$ with \eqref{bounds-tauj} and $|\partial_{t}R_{\ell}^{j}|$ with \eqref{boundRconvolution}, and we get
\begin{equation}\label{eq-bound-tildeR3}
    \begin{aligned}
        \lVert \tilde{R}_{(3)}\rVert_{C_{t}L_{x}^{1}} &\le c_{d}\Big\|\overunderset{2d}{j=1}{\sum}L^{-1}\tau_{j}\partial_{t}R_{\ell}^{j}D_{j}\Big\|_{C_{t}L_{x}^{1}}\\
        &\le c_{d}L^{-1}\sup_{t,j}\Big(\tau_{j}|\partial_{t}R_{\ell}^{j}|\Big)\\
        &\le c_{d,q}\bar{r}^{d/q'}\ell^{-d-1}\sigma^{-1}\eta^{2} 
        \leq c_{d, q} \eta^{2 - \gamma - \beta(d + 1) + \mu \frac{d}{q^\prime}} 
        < \eta^{a_4}/7.
    \end{aligned}
\end{equation}

\medskip
\noindent ($\lVert \tilde{R}_{(4)}\rVert_{C_{t}L_{x}^{1}}$). First, we use the fact that $|x-x_{j}(t)|\le r_{j}(t,x_{j}(t))$ on the support of $D_{j}(t,\cdot)$ and that $\lVert D_{j}(t,\cdot)\rVert_{L^{1}}=1$. Then we estimate the derivatives of $R_{\ell}^{j}$ with \eqref{boundRconvolution} and we use the upper bounds on $r_{j}, |\dot{x}_{j}|$ and $\tau_{j}$ from \eqref{defn_rminrmax}, \eqref{boundsspeed} and \eqref{bounds-tauj} to get 
\begin{equation}\label{eq-bound-tildeR4}
    \begin{aligned}
        \lVert \tilde{R}_{(4)}\rVert_{C_{t}L_{x}^{1}} &\le c_{d,q}L^{-1}\sup_{t,j}\Big(\tau_{j}(|\partial_{t}R_{\ell}^{j}|+|\nabla R_{\ell}^{j}||\dot{x}_{j}|)r_{j}\Big)\\
        &\le c_{d,q}\bar{r}\ell^{-d-q'-1}\eta^{1+q'/d}
        \leq c_{d, q} \eta^{\mu - \beta(d + q^\prime + 1)+ 1 + \frac{q^\prime}{d}} < \eta^{a_4}/7. 
    \end{aligned}
\end{equation}

\medskip
\noindent ($\lVert \tilde{R}_{(5)}\rVert_{C_{t}L_{x}^{1}}$). Combining Lemma \ref{lem-smart-anti-divergence} with the estimates on the derivatives of $R_{\ell}$ in \eqref{boundRconvolution} and the bound on the $L^{1}$-norm of $v_{j}$ from \eqref{propertiesvj}, we get
\begin{equation}\label{eq-bound-tildeR5}
    \begin{aligned}
        \lVert \tilde{R}_{(5)}\rVert_{C_{t}L_{x}^{1}} &\le \overunderset{2d}{j=1}{\sum}c_{d}\lVert \partial_{\xi_{j}}R_{\ell}^{j}\rVert_{C_{t}C_{x}^{1}}\lVert v_{j}\rVert_{C_{t}L_{x}^{1}}\\
         &\le c_{d}\lambda^{-1}\ell^{-d-2}\eta
         \leq c_d \eta^{\kappa - \beta(d + 2)+1}  < \eta^{a_4}/7.
\end{aligned}
\end{equation}

\medskip
\noindent ($\lVert \tilde{R}_{(6)}\rVert_{C_{t}L_{x}^{1}}$).  Since $D_{j}(t,\cdot)$ is supported in the ball of radius $r_{j}(t,x_{j}(t))$ centered in $x_{j}(t)$, the mean value theorem 
gives
 \begin{equation*}
     |\langle \psi_{j}\rangle(t,x)-\varphi_j(t,x)|\le c_{d}\Big(\lVert \nabla \partial_{\xi_{j}}R_{\ell}^{j}\rVert_{C_{t,x}}r_{\max}+\lVert \partial_{t} \partial_{\xi_{j}}R_{\ell}^{j}\rVert_{C_{t,x}}\sup_{t}\tau_{j}\Big)\frac{1}{L|\xi_{j}|}\int_{t}^{t+\tau_{j}(t)}D_{j}(s,x)|\dot{x}_{j}(s)|\,ds.
 \end{equation*}
 Then, since $\lVert D_{j}(t,\cdot)\rVert_{L^{1}}=1$ for every $t$, using the upper bounds on $r_{j}$ and $\tau_{j}$ from \eqref{defn_rminrmax} and \eqref{bounds-tauj} along with the estimates on the second order derivatives of $R_{\ell}$ from \eqref{boundRconvolution}, we get
 \begin{equation}\label{eq-bound-tildeR6}
     \begin{aligned}
     \lVert \tilde{R}_{(6)}\rVert_{C_{t}L_{x}^{1}} &  \le c_{d}\Big(\lVert \nabla^{2}R_{\ell}\rVert_{C_{t,x}}r_{\max}+\lVert \partial_{t} \nabla R_{\ell}\rVert_{C_{t,x}}\sup_{t,j}\tau_{j}\Big)\\
     &\le c_{d,q}\left( \bar{r}\ell^{-d-q'-2}\eta^{1+q'/d}+\bar{r}^{d/q'}\lambda^{d-1} \ell^{-d-2} \sigma^{-1}\eta^{2}\right) \\
     & \leq c_{d, q} \left(\eta^{\mu -  \beta (2 + q^\prime + d) + 1 +\frac{q^\prime}{d}} +\eta^{\mu \frac{d}{q^\prime}- \kappa (d-1) -  \beta (d + 2)  - \gamma +2} \right)< \eta^{a_4}/7.
 \end{aligned}
 \end{equation}

\medskip
\noindent ($\lVert \tilde{R}_{(7)}\rVert_{C_{t}L_{x}^{1}}$). This term is bounded via H\"{o}lder inequality, using assumption \eqref{assumption-density&field}, the bounds on the perturbations that we already proved in \eqref{bdd thetaP L1}, \eqref{bdd thetaT Lq}, \eqref{bdd thetaT L1}, \eqref{bdd WP W1p} and \eqref{bdd WP L1}, and the Sobolev embedding $W^{1,p}\hookrightarrow L^{q'}$: 
\begin{equation}\label{eq-bound-tildeR7}
     \begin{aligned}
        \lVert \tilde{R}_{(7)}\rVert_{C_{t}L_{x}^{1}}&\le c_{d}\left(\lVert \rho_{\ell}\rVert_{C_{t,x}}\lVert W_{P}\rVert_{C_{t}L_{x}^{1}}+ (\lVert\Theta_{P}\rVert_{C_{t}L_{x}^{1}}+\lVert\Theta_{T}\rVert_{C_{t}L_{x}^{1}}) \lVert b_{\ell}\rVert_{C_{t,x}}+\lVert \Theta_{T}\rVert_{C_{t}L_{x}^{q}}\lVert W_{P}\rVert_{C_{t}L_{x}^{q'}}\right)\\
        &\le c_{d,q}\left(\bar{r}^{d/q}\ell^{-dq'+d/q'}\sigma \eta^{1/(q-1)}+ \bar{r}^{d/q'}\lambda^{d-1}\ell^{-3d-2+d/q}\sigma^{-1}\eta^{2}+\bar{r}^{1/q'}\lambda^{(d-1)/q}\ell^{-d/(q-1)-2d-3}\eta^{1+1/d}\right)  \\
        & \leq c_{d, q} \left( \eta^{\mu \frac{d}{q} - \beta (d q^\prime - \frac{d}{q^\prime}) + \gamma + \frac{1}{q-1}} + \eta^{\mu \frac{d}{q^\prime} - \kappa(d-1) - \beta (3d+2 - \frac{d}{q}) - \gamma+2} + \eta^{\frac{\mu}{q^\prime} - \kappa \frac{(d-1)}{q} - \beta \left(\frac{d}{q-1} + 2d + 3\right) + 1 + \frac{1}{d}}\right) \\
        & < \eta^{a_4}/7.
    \end{aligned}
\end{equation}

\paragraph{Conclusion: proof of Theorem \ref{thm_pert}.} Having derived all the necessary estimates, we are now ready to give the proof of our perturbation theorem:
\begin{proof}[Proof of Theorem \ref{thm_pert}]
    Let $d\ge 3$, and $q>1$ be given. We choose $\delta=\delta(d,q)>0$ as in \eqref{formula-gap}, and we fix $p\ge 1$ such that $0\le \alpha \coloneqq  1+\frac{1}{d}-\frac{1}{p}-\frac{1}{q}<\delta$. Let $\bar\eta\in (0,1)$ be a number to be chosen sufficiently small so that all estimates derived in the present section hold true, and let the parameters $a_{1},a_{2},a_{3},a_{4},\beta,\kappa, \gamma$, and $\mu$ be given by \eqref{eq-choice-parameters}. Given $T\in [1,2]$, $\eta\in (0,\bar\eta]$, and $(\rho,b,R)\in \mathscr{A}_{a_{1},a_{2}}(T,\eta)$, we let $(\tilde \rho, \tilde b, \tilde R)$ be the triplet constructed in Section \ref{sec_constructions}, see \eqref{formulaperturbations} and \eqref{def-tildeR}, where the shape parameters $\bar r, \lambda, \ell$, and $\sigma$ are chosen as in \eqref{eq-choice-shape-parameters}. As already observed in Section \ref{sec_constructions}, $(\tilde \rho, \tilde b, \tilde R)$ solves \eqref{CDE}. Moreover, combining the estimates \eqref{eq-bound-rhoell-rho}, \eqref{bdd thetaP Lq}, \eqref{bdd thetaT Lq}, \eqref{eq-bound-Thetac}, \eqref{eq-bound-bell-b}, and \eqref{bdd WP W1p} we deduce the perturbation estimate \eqref{conclusion-perturbation-distance}. Hence, to conclude, we only need to prove that $(\tilde \rho, \tilde b, \tilde R)$ belongs to $\mathscr{A}_{a_{1},a_{2}}(T-\eta^{a_{3}},\eta^{a_{4}})$. First, combining \eqref{assumption-density&field} with \eqref{conclusion-perturbation-distance}, we deduce the upper bound on the new density and field \eqref{conclusion-density&field}, up to choosing $\bar\eta$ sufficiently small. Then, by collecting \eqref{eq-bound-der-rhoell}, \eqref{eq-bound-der-Thetap}, \eqref{eq-bound-tder-Thetat}, \eqref{eq-bound-xder-Thetat}, \eqref{eq-bound-tder-Thetac}, \eqref{eq-bound-der-bell}, and \eqref{eq-bound-der-Wp}, we obtain the estimate on derivatives \eqref{conclusion-higher-order-derivatives-density&field}. Finally, combining \eqref{eq-bound-tildeR1}, \eqref{eq-bound-tildeR2}, \eqref{eq-bound-tildeR3}, \eqref{eq-bound-tildeR4}, \eqref{eq-bound-tildeR5}, \eqref{eq-bound-tildeR6}, and \eqref{eq-bound-tildeR7}, we get the new error bound \eqref{conclusion-error}. This concludes the proof. 
\end{proof}

\paragraph{Acknowledgments.} We thank Elia Bruè for the many useful conversations about the topic of this paper. M.C. is supported by the Swiss State Secretariat for Education, Research and Innovation (SERI) under contract number MB22.00034 through the project TENSE. R.C. is supported by the Swiss National Science Foundation (SNF grant
PZ00P2\_208930).

\bigskip

\noindent \textbf{Disclosure statement.} The authors report there are no competing interests to declare.

\bigskip

\noindent \textbf{Data availability statement.} No datasets were generated or analyzed in this theoretical study. Data sharing is therefore not applicable.

\bibliographystyle{halpha_abbrv}
\bibliography{Refs_CI_Transport_2026_revision}

@article{MoSz2019AnnPDE,
  author  = {Modena, Stefano and Sz{\'e}kelyhidi, Jr., L{\'a}szl{\'o}},
  title   = {Non-uniqueness for the transport equation with {S}obolev vector fields},
  journal = {Ann. PDE},
  volume  = {4},
  number  = {2},
  pages   = {Paper No. 18, 38},
  year    = {2018},
  doi     = {10.1007/s40818-018-0056-x},
  mrnumber = {3884855}
}

@article{MoSa2019,
  author  = {Modena, Stefano and Sattig, Gabriel},
  title   = {Convex integration solutions to the transport equation with full dimensional concentration},
  journal = {Ann. Inst. H. Poincar\'e C Anal. Non Lin\'eaire},
  volume  = {37},
  number  = {5},
  pages   = {1075--1108},
  year    = {2020},
  doi     = {10.1016/j.anihpc.2020.03.002},
  mrnumber = {4138227}
}

@article{BrueColomboDeLellis21,
  author  = {Bru{\`e}, Elia and Colombo, Maria and De Lellis, Camillo},
  title   = {Positive solutions of transport equations and classical nonuniqueness of characteristic curves},
  journal = {Arch. Ration. Mech. Anal.},
  volume  = {240},
  number  = {2},
  pages   = {1055--1090},
  year    = {2021},
  doi     = {10.1007/s00205-021-01628-5},
  mrnumber = {4244826}
}

@article{DiPernaLions,
  author  = {DiPerna, Ronald J. and Lions, Pierre-Louis},
  title   = {Ordinary differential equations, transport theory and {S}obolev spaces},
  journal = {Invent. Math.},
  volume  = {98},
  number  = {3},
  pages   = {511--547},
  year    = {1989},
  doi     = {10.1007/BF01393835},
  mrnumber = {1022305}
}

@article{de2009euler,
  author  = {De Lellis, Camillo and Sz{\'e}kelyhidi, Jr., L{\'a}szl{\'o}},
  title   = {The {E}uler equations as a differential inclusion},
  journal = {Ann. of Math. (2)},
  volume  = {170},
  number  = {3},
  pages   = {1417--1436},
  year    = {2009},
  doi     = {10.4007/annals.2009.170.1417},
  mrnumber = {2600877}
}

@article{de2013dissipative,
  author  = {De Lellis, Camillo and Sz{\'e}kelyhidi, Jr., L{\'a}szl{\'o}},
  title   = {Dissipative continuous {E}uler flows},
  journal = {Invent. Math.},
  volume  = {193},
  number  = {2},
  pages   = {377--407},
  year    = {2013},
  doi     = {10.1007/s00222-012-0429-9},
  mrnumber = {3090182}
}

@article{isett2018proof,
  author  = {Isett, Philip},
  title   = {A proof of {O}nsager's conjecture},
  journal = {Ann. of Math. (2)},
  volume  = {188},
  number  = {3},
  pages   = {871--963},
  year    = {2018},
  doi     = {10.4007/annals.2018.188.3.4},
  mrnumber = {3866888}
}

@article{buckmaster2019onsager,
  author  = {Buckmaster, Tristan and De Lellis, Camillo and Sz{\'e}kelyhidi, Jr., L{\'a}szl{\'o} and Vicol, Vlad},
  title   = {{O}nsager's conjecture for admissible weak solutions},
  journal = {Comm. Pure Appl. Math.},
  volume  = {72},
  number  = {2},
  pages   = {229--274},
  year    = {2019},
  doi     = {10.1002/cpa.21781},
  mrnumber = {3896021}
}

@article{buckmaster2019nonuniqueness,
  author  = {Buckmaster, Tristan and Vicol, Vlad},
  title   = {Nonuniqueness of weak solutions to the {N}avier--{S}tokes equation},
  journal = {Ann. of Math. (2)},
  volume  = {189},
  number  = {1},
  pages   = {101--144},
  year    = {2019},
  doi     = {10.4007/annals.2019.189.1.3},
  mrnumber = {3898708}
}

@article{buckmaster2021wild,
  author  = {Buckmaster, Tristan and Colombo, Maria and Vicol, Vlad},
  title   = {Wild solutions of the {N}avier--{S}tokes equations whose singular sets in time have {H}ausdorff dimension strictly less than 1},
  journal = {J. Eur. Math. Soc. (JEMS)},
  volume  = {24},
  number  = {9},
  pages   = {3333--3378},
  year    = {2022},
  doi     = {10.4171/JEMS/1162},
  mrnumber = {4482534}
}

@article{MR4199851,
  author  = {Cheskidov, Alexey and Luo, Xiaoyutao},
  title   = {Nonuniqueness of weak solutions for the transport equation at critical space regularity},
  journal = {Ann. PDE},
  volume  = {7},
  number  = {1},
  pages   = {Paper No. 2, 45},
  year    = {2021},
  doi     = {10.1007/s40818-020-00091-x},
  mrnumber = {4199851}
}

@article{cheskidov2022extreme,
  author  = {Cheskidov, Alexey and Luo, Xiaoyutao},
  title   = {Extreme temporal intermittency in the linear {S}obolev transport: almost smooth nonunique solutions},
  journal = {Anal. PDE},
  volume  = {17},
  number  = {6},
  pages   = {2161--2177},
  year    = {2024},
  doi     = {10.2140/apde.2024.17.2161},
  mrnumber = {4776296},
  eprint  = {2204.08950},
  archivePrefix = {arXiv}
}

@article{CrDL,
  author  = {Crippa, Gianluca and De Lellis, Camillo},
  title   = {Estimates and regularity results for the {D}i{P}erna--{L}ions flow},
  journal = {J. Reine Angew. Math.},
  volume  = {616},
  pages   = {15--46},
  year    = {2008},
  doi     = {10.1515/CRELLE.2008.016},
  mrnumber = {2369485}
}

@article{Ambrosio04,
  author  = {Ambrosio, Luigi},
  title   = {Transport equation and {C}auchy problem for {$BV$} vector fields},
  journal = {Invent. Math.},
  volume  = {158},
  number  = {2},
  pages   = {227--260},
  year    = {2004},
  doi     = {10.1007/s00222-004-0367-2},
  mrnumber = {2096794}
}

@article{AmbrCr,
  author  = {Ambrosio, Luigi and Crippa, Gianluca},
  title   = {Continuity equations and {ODE} flows with non-smooth velocity},
  journal = {Proc. Roy. Soc. Edinburgh Sect. A},
  volume  = {144},
  number  = {6},
  pages   = {1191--1244},
  year    = {2014},
  doi     = {10.1017/S0308210513000085},
  mrnumber = {3283066}
}

@article{kumar2023nonuniqueness,
  author  = {Kumar, Anuj},
  title   = {Nonuniqueness of trajectories on a set of full measure for {S}obolev vector fields},
  journal = {Arch. Ration. Mech. Anal.},
  volume  = {248},
  number  = {6},
  pages   = {Paper No. 114},
  year    = {2024},
  doi     = {10.1007/s00205-024-02063-y},
  eprint  = {2301.05185},
  archivePrefix = {arXiv}
}

@article{PitchoSorella2023,
  author  = {Pitcho, Jules and Sorella, Massimo},
  title   = {Almost everywhere nonuniqueness of integral curves for divergence-free {S}obolev vector fields},
  journal = {SIAM J. Math. Anal.},
  volume  = {55},
  number  = {5},
  pages   = {4640--4663},
  year    = {2023},
  doi     = {10.1137/22M1487187},
  mrnumber = {4634742},
  eprint  = {2108.03194},
  archivePrefix = {arXiv}
}

@article{GiriSorella21,
  author  = {Giri, Vikram and Sorella, Massimo},
  title   = {Non-uniqueness of integral curves for autonomous {H}amiltonian vector fields},
  journal = {Differential Integral Equations},
  volume  = {35},
  number  = {7--8},
  pages   = {411--436},
  year    = {2022},
  doi     = {10.57262/die035-0708-411},
  mrnumber = {4412812},
  eprint  = {2108.05050},
  archivePrefix = {arXiv}
}

@article{pappalettera2023measure,
  author  = {Pappalettera, Umberto},
  title   = {On measure-preserving selection of solutions of {ODE}s},
  journal = {Proc. Amer. Math. Soc.},
  volume  = {153},
  number  = {5},
  pages   = {2037--2051},
  year    = {2025},
  doi     = {10.1090/proc/17097},
  eprint  = {2311.15661},
  archivePrefix = {arXiv}
}

@article{Jabin2010,
  author  = {Jabin, Pierre-Emmanuel},
  title   = {Differential equations with singular fields},
  journal = {J. Math. Pures Appl. (9)},
  volume  = {94},
  number  = {6},
  pages   = {597--621},
  year    = {2010},
  doi     = {10.1016/j.matpur.2010.07.001},
  mrnumber = {2733801}
}

@article{CaravennaCrippa16,
  author  = {Caravenna, Laura and Crippa, Gianluca},
  title   = {Uniqueness and {L}agrangianity for solutions with lack of integrability of the continuity equation},
  journal = {C. R. Math. Acad. Sci. Paris},
  volume  = {354},
  number  = {12},
  pages   = {1168--1173},
  year    = {2016},
  doi     = {10.1016/j.crma.2016.10.009},
  mrnumber = {3566538},
  eprint  = {1608.04324},
  archivePrefix = {arXiv}
}

@article{CaravennaCrippa18,
  author  = {Caravenna, Laura and Crippa, Gianluca},
  title   = {A directional {L}ipschitz extension lemma, with applications to uniqueness and {L}agrangianity for the continuity equation},
  journal = {Comm. Partial Differential Equations},
  volume  = {46},
  number  = {8},
  pages   = {1488--1520},
  year    = {2021},
  doi     = {10.1080/03605302.2021.1883650},
  mrnumber = {4286465},
  eprint  = {1812.06817},
  archivePrefix = {arXiv}
}

@article{brue2024sharp,
  author  = {Bru{\`e}, Elia and Colombo, Maria and Kumar, Anuj},
  title   = {Sharp nonuniqueness in the transport equation with {S}obolev velocity field},
  journal = {J. Eur. Math. Soc.},
  year    = {2026},
  doi     = {10.4171/JEMS/1756},
  note    = {Published online first},
  eprint  = {2405.01670},
  archivePrefix = {arXiv}
}

@article{BrCoKu24,
author = {Bru{\`e}, Elia and Colombo, Maria and Kumar, Anuj},
title = {Flexibility of two-dimensional {E}uler flows with integrable vorticity},
journal = {Duke Math. J.},
volume = {175},
number = {9},
pages = {1593--1647},
year = {2026},
doi = {10.1215/00127094-2025-0056}
}

@article{BrueColombo23,
  author  = {Bru{\`e}, Elia and Colombo, Maria},
  title   = {Nonuniqueness of solutions to the {E}uler equations with vorticity in a {L}orentz space},
  journal = {Comm. Math. Phys.},
  volume  = {403},
  number  = {2},
  pages   = {1171--1192},
  year    = {2023},
  doi     = {10.1007/s00220-023-04816-4},
  mrnumber = {4620753}
}

@article{BuckModena24,
  author  = {Buck, Miriam and Modena, Stefano},
  title   = {Non-uniqueness and energy dissipation for {2D} {E}uler equations with vorticity in {H}ardy spaces},
  journal = {J. Math. Fluid Mech.},
  volume  = {26},
  pages   = {Paper No. 26},
  year    = {2024},
  doi     = {10.1007/s00021-024-00860-9},
  eprint  = {2306.05948},
  archivePrefix = {arXiv}
}

@article{bogovskii1979solution,
  author  = {Bogovski{\u{\i}}, Mikhail Evgen'evich},
  title   = {Solution of the first boundary value problem for the equation of continuity of an incompressible medium},
  journal = {Dokl. Akad. Nauk SSSR},
  volume  = {248},
  number  = {5},
  pages   = {1037--1040},
  year    = {1979},
  mrnumber = {0559207}
}

@article{MoSz2019CalcVar,
  author  = {Modena, Stefano and Sz{\'e}kelyhidi, Jr., L{\'a}szl{\'o}},
  title   = {Non-renormalized solutions to the continuity equation},
  journal = {Calc. Var. Partial Differential Equations},
  volume  = {58},
  number  = {6},
  pages   = {Paper No. 208, 30},
  year    = {2019},
  doi     = {10.1007/s00526-019-1651-8}
}

\end{document}